\documentclass[10pt]{amsart}

\usepackage[left=3cm,right=3cm,
    top=3cm,bottom=3cm,bindingoffset=0cm]{geometry}

\usepackage{hyperref}
\usepackage{amsaddr}
\usepackage{amsthm,amssymb,amsmath,amsfonts}
\usepackage{graphicx}
\usepackage{pgf,tikz,pgfplots,tikz-cd}
\pgfplotsset{compat=1.15}
\usepackage{mathrsfs}
\usepackage{ytableau}
\usepackage{multirow}

\usetikzlibrary{arrows}
\allowdisplaybreaks

\newcommand{\rge}{\rangle}
\newcommand{\lge}{\langle}

\newcommand{\mc}{\mathcal}
\newcommand{\spec}{\operatorname{Spec}}
\newcommand{\op}{\operatorname{Op}}
\renewcommand{\le}{\leqslant}
\renewcommand{\ge}{\geqslant}

\newcommand{\fullstopbelow}{\makebox[0pt][l]{\,.}}
\newcommand{\commabelow}{\makebox[0pt][l]{\,,}}

\newcommand{\R}{\mathbb R}
\newcommand{\C}{\mathbb C}

\newcommand{\Z}{\mathbb Z}

\renewcommand{\b}{\mathfrak b}

\newcommand{\n}{\mathfrak n}
\newcommand{\g}{\mathfrak g}
\newcommand{\h}{\mathfrak h}
\newcommand{\jet}[1]{ [\hspace{-0,5mm}[ {#1} ]\hspace{-0,5mm}] }
\newcommand{\rjet}[1]{ (\hspace{-0,7mm}( {#1} )\hspace{-0,7mm}) }
\newcommand{\git}{/\!\!/}

\numberwithin{equation}{section}
\theoremstyle{plain}
\newtheorem{theorem}{Theorem}[section]

\newtheorem{definition}[theorem]{Definition}
\newtheorem{proposition}[theorem]{Proposition}
\newtheorem{corollary}[theorem]{Corollary}

\newtheorem{lemma}[theorem]{Lemma}

\newtheorem{remark}[theorem]{Remark}
\newtheorem{example}[theorem]{Example}

\newtheorem{conjecture}[theorem]{Conjecture}

\title[Dynkin automorphism actions on Gaudin algebras]{Dynkin automorphism actions on Gaudin algebras}
\author[Vladyslav Zveryk]{Vladyslav Zveryk}
\address{Faculty of Mathematics and Computer Science\\Jagiellonian University in Kraków\\Łojasiewicza 6, 30-348 Kraków, Poland\\ \mbox{ }\\
{\bfseries Current address}:\\
Department of Mathematics\\
Yale University\\
219 Prospect Street, CT 06511, New Haven, USA}
\email{zverik.vladislav@gmail.com}
\begin{document}
\maketitle
\begin{abstract}

We study the action of Dynkin diagram automorphisms $\sigma$ on generalized Gaudin algebras, focusing in particular on the big Gaudin algebra \( \mathcal{B}(\mathfrak{g}) \subset (U(\mathfrak{g}) \otimes S(\mathfrak{g}))^{\mathfrak{g}} \) and its evaluated versions \( \mathcal{B}^\lambda(\mathfrak{g}) \) and \( \mathcal{B}_\chi(\mathfrak{g}) \). We show isomorphisms between the coinvariants of the generalized Gaudin algebras associated with $\g^\vee$ and the generalized Gaudin algebras associated with $\g_\sigma^\vee$, where $\g_\sigma$ is the fixed point subalgebra. In particular, we get an isomorphism \( \mathcal{B}^\lambda(\mathfrak{g}^\vee)_\sigma \cong \mathcal{B}^\lambda(\mathfrak{g}^\vee_\sigma) \) for any \( \sigma \)-invariant dominant weight \( \lambda \), which allows us to reprove Jantzen’s twining formula. Our approach relies on interpreting generalized Gaudin algebras via spaces of opers, which explains the appearance of the Langlands duals in our results and in Jantzen's twining formula.

    % We study the action of a Dynkin automorphism $\sigma$ of a simple Lie algebra $\g$ on the generalized Gaudin algebras. In particular, we show isomorphisms between the coinvariants of the generalized Gaudin algebras associated with $\g^\vee$ and the generalized Gaudin algebras associated with $\g_\sigma^\vee$, where $\g_\sigma$ is the fixed point subalgebra. We apply our results to reprove Jantzen's twining formula.
\end{abstract}
\section{Introduction}
Let $\g$ be a simple Lie algebra over $\C$. One can give a series of constructions giving various commutative algebras associated to $\g$ called generalized Gaudin algebras (\cite{FFRe}, \cite{FFT}). They attract great attention by arising in quantum integrable systems and geometric Langlands correspondence \cite{FFT}. 

In this paper, we take a closer look at a particular Gaudin algebra $\mc B(\g)\subset (U(\g)\otimes S(\g))^\g$ and its evaluations $\mc B^\lambda(\g)\subset(\operatorname{End}_\C V(\lambda)\otimes S(\g))^\g$ at irreducible $\g$-representations $V(\lambda)$ of highest weights $\lambda$ \cite{Ha}. The algebra $\mc B(\g)$ plays a crucial role in the answer to Vinberg's quantization problem \cite[Section 9.2]{Mo}, parametrizing its solutions at regular elements, while the algebra $\mc B^\lambda(\g)$ is expected to possess significant information about the representation $V(\lambda)$, as explained in \cite{Ha}. When $V(\lambda)$ is a weight multiplicity free representation, the algebra $(\operatorname{End}_\C V(\lambda)\otimes S(\g))^\g$ is commutative itself and is known to encode the structure of $V(\lambda)$, as described in detail in \cite{Pa}. In general, it can be shown that  $B^\lambda(\g)$ is the maximal commutative subalgebra of $(\operatorname{End}_\C V(\lambda)\otimes S(\g))^\g$, so there is a general idea of Tam\'{a}s Hausel in \cite{Ha} that it is a natural candidate for the corresponding algebra for any $\lambda$.

Our goal is to study the actions of Dynkin automorphisms on generalized Gaudin algebras, in particular, on $\mc B(\g)$ and $\mc B^\lambda(\g)$. The main motivation for our work is the following conjecture:

\begin{conjecture}[Conjecture 4.1, \cite{Ha}]\label{introconj}
    There exists an isomorphism 
    $$
    \mc B^\lambda(\g^\vee)_\sigma\simeq \mc B^\lambda(\g_\sigma^\vee),
    $$
    where $\g_\sigma$ is the subalgebra of $\g$ fixed by $\sigma$, $\g^\vee$ and $\g_\sigma^\vee$ are the corresponding Langlands duals, and $\mc B^\lambda(\g^\vee)_\sigma$ denotes the quotient of $\mc B^\lambda(\g^\vee)$ by the ideal generated by $x-\sigma(x)$ for all $x\in \mc B^\lambda(\g^\vee)$.
\end{conjecture}

This conjecture, in turn, is motivated by Jantzen's twining formula (\cite{Ja}, \cite{HS}) which asserts that for any $\sigma$-invariant dominant weight $\lambda$ and any $\sigma$-invariant weight $\mu$ of $\g^\vee$ we have
$$
    \mathrm{tr}(\sigma|V_\mu(\lambda))=\dim W_\mu(\lambda),
    $$
where $V_\mu(\lambda)$ and $W_\mu(\lambda)$ denote the corresponding weight spaces of the corresponding highest weight representations of $\g^\vee$ and $\g_\sigma^\vee$, respectively. It is expected in \cite{Ha} that Conjecture \ref{introconj} implies this formula.

We give a proof of Conjecture \ref{introconj}, as well as of similar isomorphisms for other types of generalized Gaudin algebras and the Feigin-Frenkel center (\cite{FFRe}, \cite{FFR},\cite{Mo},\cite{FFT}). We also support our results by arguments on the naturality of the isomorphism constructed and derive Jantzen's twining formula from this result.

The paper is organized as follows. Sections 2-4 are devoted to preliminaries, covering all the necessary information about the generalized Gaudin algebras, Dynkin automorphisms and opers. We included more information than is needed for the main results of the paper to give a broader introduction to the theory, supporting it with relevant references for further reading.

In Section 5, the actions of Dynkin automorphisms on the spaces of opers are described, while in Section 6 all the information presented is combined to produce isomorphisms between some generalized Gaudin algebras associated to $\g^\vee$ and $\g_\sigma^\vee$, including the Feigin-Frenkel center. 

In Section 7, we take a closer look at the constructed isomorphisms, showing the compatibility of these isomorphisms with representation theories of $\g^\vee$ and $\g_\sigma^\vee$.

In Section 8, we finally prove Conjecture \ref{introconj} and derive Jantzen's twining formula from it.
\section{Generalized Gaudin algebras}\label{s:gaudinalg}
We refer to \cite[Section 2]{FFT} for details.
\subsection{Feigin-Frenkel center}
Let $\g$ be a simple Lie algebra. Let $\kappa$ be the {\bfseries critical level}, which is minus one half times the Killing form of $\g$. The affine Kac-Moody algebra $\widehat \g$ is the extension of $\g\rjet t=\g\otimes \C\rjet t$ by the one-dimensional center $\C\mathbf 1$:
$$
0\to \C\mathbf 1\to \widehat\g\to \g\rjet t\to 0,
$$
with the Lie bracket
$$
[A[m],B[n]]=[A,B][m+n]-n\delta_{-m,n}\kappa(A,B)\cdot \mathbf 1,
$$
where $A[m]$ denotes $A\otimes t^m$ for $A\in\g$ and $m\in\Z$. This formula can be equivalently written as
$$
[A\otimes f(t),B\otimes g(t)]=[A,B]\otimes fg-\kappa(A,B)\operatorname{Res}_{t=0}fg'\cdot \mathbf 1.
$$

Let $\widehat \g_+$ be the Lie subalgebra $\g\jet t\oplus \C\mathbf 1$ of $\widehat \g$. Define the $\widehat\g$-module
$$
\mathbb V_0:=\operatorname{Ind}_{\widehat \g_+}^{\widehat \g}\C=U(\widehat \g)\otimes_{U(\widehat \g_+)}\C,
$$
where $\C$ is the $\widehat\g_+$-module on which $\g\jet t$ acts trivially and $\mathbf 1$ acts as the identity. Note that the element $\mathbf 1:=1\otimes 1\in \mathbb V_0$ is a highest weight vector of this $\widehat \g$-representation annihilated by $\g\jet t$. The space $\mathbb V_0$ has a vertex algebra structure, with which $\mathbb V_0$ is called the {\bfseries universal affine vertex algebra associated with $\g$ (and $\kappa$)}.  For more details, the reader may consult \cite[Section 2.4]{FBZ} or \cite[Section 6.2]{Mo}, but we will not use the vertex algebras approach in this paper.
\begin{definition}
    The center of $\mathbb V_0$ as a vertex algebra is denoted by $\mathfrak z(\widehat \g)$ and is called the {\bfseries Feigin-Frenkel center}.
\end{definition}

Denote by $\widehat\g_-$ the subalgebra $t^{-1}\g[t^{-1}]$ of $\widehat \g$. By the Poincar\'{e}-Birkhoff-Witt theorem, we get an isomorphism of $\widehat\g_-$-modules
$$
U(\widehat\g_-)\simeq U(\widehat \g)\otimes_{U(\widehat \g_+)}\C=\mathbb V_0
$$
induced by the inclusion $U(\widehat\g_-)\hookrightarrow U(\widehat \g)$, where $\widehat\g_-$ acts on both sides by left multiplication. Let $X_1,\ldots,X_n$ and $X^1,\ldots,X^n$ be bases of $\g$ dual with respect to $\kappa$. 
\begin{theorem}[Equivalent descriptions of the Feigin-Frenkel center]\
\begin{itemize}
    \item[(i)] It is the center of the vertex algebra $\mathbb V_0$.
    \item[(ii)] It equals $\mathbb V_0^{\g\jet t}$ as a vector subspace of $\mathbb V_0$.
    \item[(iii)] After the vector space identification $\mathbb V_0\simeq U(\widehat\g_-)$, it is the subalgebra of $U(\widehat\g_-)$ equal to the centralizer of $\sum_{i}X_i[-1]X^i[-1]$.
    \item[(iv)] The inclusion $\mathfrak z(\widehat \g)\subset U(\widehat\g_-)^{\mathrm{op}}$ as algebras is isomorphic to the inclusion of algebras
    $$
    \operatorname{End}_{\widehat \g}\mathbb V_0\subset \operatorname{End}_{\widehat \g_-}\mathbb V_0.
    $$
\end{itemize}    
\end{theorem}
\begin{proof}
    Item $(i)$ is the definition. Item $(ii)$ follows from \cite{Mo}, the discussion after Definition 6.2.1. Item $(iii)$ follows from \cite[Theorem 1]{R2}. 

    The identifications in $(iv)$ are done as follows. Endomorphisms of $\mathbb V_0$ which are $\widehat\g$-invariant are completely determined by images of the highest weight vector $\mathbf 1\in \mathbb V_0$, and these images can be chosen arbitrarily from the set of $\g\jet t$-invariant elements of  $\mathbb V_0$. This gives an isomorphism 
    $$
    \mathfrak z(\widehat \g)\simeq \operatorname{End}_{\widehat \g}\mathbb V_0
    $$
    of vector spaces. Similarly, elements of $U(\widehat\g_-)^{\mathrm{op}}$ act on $\mathbb V_0\simeq U(\widehat\g_-)$ by right multiplication, whence the isomorphism
    $$
    U(\widehat\g_-)^{\mathrm{op}}\simeq \operatorname{End}_{\widehat \g_-}\mathbb V_0.
    $$
   
\end{proof}

The Feigin-Frenkel center possesses a beautiful structure. First of all, there are derivations $\tau=-\partial_t$ and $t\tau=-t\partial_t$ on $U(\widehat\g_-)$, which restrict to derivations on $\mathfrak z(\widehat \g)$. The derivation $t\tau$ acts on $U(\widehat\g_-)$ diagonalizably with integral eigenvalues, hence induces gradings on $U(\widehat\g_-)$ and $\mathfrak z(\widehat \g)$.

Let us recall the structure of the $G$-invariant polynomials on $\g$:
\begin{theorem}[Chevalley's theorem, \cite{Ko}, \cite{Ch}]\label{t:Chevalley}
    Let $\h$ be a Cartan subalgebra of a reductive Lie algebra $\g$ and $W$ be the corresponding Weyl group. Then the restriction map $S(\g^*)^G\to S(\h^*)^W$ is an isomorphism. Moreover, there exist algebraically independent homogeneous polynomials $P_1,\ldots, P_l\in S(\g^*)^G$ that generate $S(\g^*)^G$. For any such choice of $P_1,\ldots, P_l$ of respective degrees $d_1\le d_2\le\ldots\le d_l$ the following is true:
    \begin{enumerate}
        \item $l=\mathrm{rk}\,\g$ and $d_i$ don't depend on the choice of generators.
        \item $\sum_id_i=\dim \b$, where $\b$ is any Borel subalgebra of $\g$.
    \end{enumerate}
\end{theorem}

The generalization of this result for $\mathfrak z(\widehat\g)$ is the following: 
\begin{definition}
    Let $l$ be the rank of $\g$. A set of homogeneous elements $S_1,\ldots, S_l\in \mathfrak z(\widehat \g)$ is called a {\bfseries complete set of Segal-Sugawara vectors} if $\mathfrak z(\widehat \g)$ is a polynomial algebra over $\C$ freely generated by $\tau^kS_i$ for $i=1\ldots l$ and $k\ge 0$.
\end{definition}
\begin{theorem}[Feigin-Frenkel theorem]\label{t:FF}
    A complete set $S_1,\ldots,S_l$ of Segal-Sugawara vectors exists for every reductive Lie algebra $\g$. Moreover, the list of degrees of $S_1,\ldots,S_l$ coincides with the list of degrees of homogeneous invariant generators of $S(\g)^\g$.
\end{theorem}
\begin{proof}
    Follows from the discussion before (5.3) in \cite{FFT}.
\end{proof}

The Segal-Sugawara vector $S_1$ of the smallest degree will be of particular interest to us. First of all, there is an element in $S(\g^*)^G$ of degree $2$, unique up to a scalar, called the {\bfseries Casimir element}. It is defined as
$$
c:=\sum_{i=1}^nX_iX^i.
$$
Then Theorem \ref{t:FF} implies that $S_1$ is the unique element of degree $2$ among all Segal-Sugawara operators. Thus, if there is another nonzero element $S_1'\in\mathfrak z(\widehat \g)$ of degree $2$, then $S_1$ is a scalar multiple of $S_1'$. The element 
$$
\sum X_i[-1]X^i[-1]\in U(t^{-1}\g[t^{-1}])
$$
is $\g\jet t$-invariant and of degree $2$, implying that $S_1$ must be its non-zero multiple.
\subsection{Generalized Gaudin algebras}
Let $z_1,\ldots,z_N\in\mathbb P^1$. Define $\widetilde\g(z_i):=\g\rjet{t-z_i}$ and $\widetilde\g(\infty):=\g\rjet{t^{-1}}$. Let $\widehat \g(z_i)$ and $\widehat\g(\infty)$ be the corresponding central extensions of these algebras by the critical level as in the case $\tilde \g(0)=\g\rjet t$. Define also
$$
\widetilde\g_{(z_i),\infty}:=\bigoplus_i\widetilde\g(z_i)\oplus \widetilde\g(\infty).
$$

Now, for $u\in \mathbb P^1$ we can define the {\bfseries Feigin-Frenkel center at $u$} to be the same construction, but associated to $\widehat\g(u)$. We denote it by $\mathfrak z_u(\widehat g)$. In particular, the Feigin-Frenkel center defined in the previous section corresponds to $u=0$. Note that all the constructions are isomorphic by a shift $t\mapsto t+u$ or by sending $t\mapsto t^{-1}$ in case $u=\infty$.

Now, for $u,z\in\mathbb P^1$, $u\ne\infty$ we define a map
$$
\rho_{u,z}:\widetilde\g(u)_-\to \widetilde\g(z),
$$
sending a polynomial in $(t-u)^{-1}$ with coefficients in $\g$ to its Laurent series expansion at $t=z$. These are well-defined maps of Lie algebras. Define
$$
\tau_{u;(z_i),\infty}:\widehat\g(u)_-=\widetilde\g(u)_-\to \widetilde \g_{(z_i),\infty}
$$
to be the diagonal map $(\rho_{u,z_1},\ldots, \rho_{u,z_N},\rho_{u,\infty})$. Note that if $z\ne \infty$, then the image of $\rho_{u,z}$ is contained in $\g\jet{t-z}$ and the image of $\rho_{u,\infty}$ is contained in $t^{-1}\g\jet{t^{-1}}$.  Then, identifying $\g\jet{t-z_i}$ with $\g\jet t$ and $t^{-1}\g\jet{t^{-1}}$ with $t\g\jet{t}$, we define an algebra anti-homomorphism
$$
\Psi_{u;(z_i),\infty}=U(-\tau_{u;(z_i),\infty}):U(\widehat \g(u)_-)\to U(\g\jet t)^{\otimes N}\otimes U(t\g\jet t).
$$
\begin{remark}
{\rm
    Note the difference between our map $\Psi_{u;(z_i),\infty}$ and the map $\Phi_{u;(z_i)}$ from \cite[Section 2.5]{FFT}: The map $\Phi$ is the composition of $\Psi$ with the isomorphism $\widehat \g_-$ with $\widehat \g(u)_-$ given by the shift $t\mapsto t-u$.
}
\end{remark}

\begin{definition}
    The {\bfseries universal Gaudin algebra}
    $$
    \mathcal Z_{(z_i),\infty}(\g)\subset U(\g\jet t)^{\otimes N}\otimes U(t\g\jet t)
    $$
    is the image of $\mathfrak z_u(\widehat\g)$ under $\Psi_{u;(z_i),\infty}$.
\end{definition}
\begin{theorem}
    The algebra $\mathcal Z_{(z_i),\infty}(\g)$ consists of $\g$-invariant elements and is independent of $u$.
\end{theorem}
\begin{proof}
    As $\mathfrak z_u(\widehat\g)$ consists of $\g$-invariant elements and the map $\Psi_{u;(z_i),\infty}$ is $\g$-equivariant, the first statement is clear. The second statement follows from the discussion after Proposition 2.8 in \cite{FFT}.
\end{proof}

\begin{definition}
    Let $m_1,\ldots,m_N$ and $m_\infty$ be positive integers. 
    \begin{itemize}
        \item The {\bfseries generalized Gaudin algebra} $\mathcal Z_{(z_i),\infty}^{(m_i),m_\infty}(\g)$ is the image of $\mathcal Z_{(z_i),\infty}(\g)$ under the quotient map
    $$
    U(\g\jet t)^{\otimes m}\otimes U(t\g\jet t)\to \bigotimes_i U(\g[t]/t^{m_i}\g[t])\otimes U(t\g[t]/t^{m_\infty}\g[t]).
    $$
    % \item In the particular case $m_1=\ldots=m_N=1$ and $m_{\infty}=1$, we call 
    % $$
    % \mathcal Z_{(z_i),\infty}^{(1),1}(\g)\subset U(\g)^{\otimes N}
    % $$
    % the {\bfseries Gaudin algebra}.
    \item Let $z_1\ne z_2$. We call 
    $$
    \mc G(z_1,z_2):=\mathcal Z_{(z_1,z_2),\infty}^{(1),1}(\g)\subset U(\g)\otimes U(\g)
    $$
    the {\bfseries Gaudin algebra}.
    \item The algebra 
    $$
    \mc Z_{z_1,\infty}^{1,2}(\g)\subset U(\g)\otimes S(\g)
    $$
    will be called the {\bfseries universal big algebra} and denoted by $\mc B(\g)$.
    \item For $\chi\in \g^*$ we denote by $\mc B_\chi(\g)$ the image of $\mc B(\g)$ under the evaluation homomorphism 
    $$
    U(\g)\otimes S(\g)\xrightarrow{1\otimes \mathrm{ev}_\chi}U(\g).
    $$
    \item For a dominant weight $\mu$ we denote by $\mc B^\mu(\g)$ the image of $\mc B(\g)$ under the homomorphism 
    $$
    U(\g)\otimes S(\g)\xrightarrow{\pi_\mu\otimes 1}\operatorname{End}_\C V(\mu)\otimes S(\g),
    $$
    where $V(\mu)$ is the highest wight representation of $\g$ of highest weight $\mu$ and $\pi_\lambda:U(\g)\to\operatorname{End}_\C V(\mu)$ is the corresponding map. The algebras $\mc B^\mu(\g)$ are called {\bfseries big algebras}.
    \end{itemize}
\end{definition}
\begin{proposition}\label{prop:genpropofgaudin}\
    \begin{itemize}
        \item[(i)] The homomorphism $\Psi_{0;z,\infty}$ takes the form
            $$
                \Psi_{0;z,\infty}(X[r])=z^rX\otimes 1+\delta_{r,-1}\otimes X
            $$
            for $X\in \g$.
        \item[(ii)] The universal big algebra $\mc Z_{z,\infty}^{1,2}(\g)$ and the Gaudin algebra $\mc G(z_1,z_2)$ are independent of the chosen points $z\ne 0$ and $z_1\ne z_2$ chosen.
        \item[(iii)] Let $S_1,\ldots,S_l$ be a complete set of Segal-Sugawara vectors of $\mathfrak{z}(\widehat\g)$. The universal big algebra is a polynomial algebra over $\C$ freely generated by the images of $\tau^k S_i$, where $i=1\ldots l$, $k=0\ldots \deg S_i$.
        \item[(iv)] The universal big algebra equals the associated graded ring of the Gaudin algebra, where the filtration on $\mc G(z_1,z_2)\subset U(\g)\otimes U(\g)$ is induced from the rightmost $U(\g)$.
    \end{itemize}
\end{proposition}
\begin{proof}
   Item $(i)$ is a simple computation. Items $(ii)-(iv)$ follow from \cite[Section 8.2]{Ya}. 
\end{proof}
Gaudin algebras attract significant attention because of their connection to various topics. For instance, the universal big algebra turns out to parameterize the quantizations of Fomenko-Mishchenko subalgebras. This property is summarized in the following lines.

Let $\g$ be a reductive Lie algebra and $P_1,\ldots,P_l$ be homogeneous algebraically independent generators of $S(\g)^\g$. For $\chi\in \g^*$, define the corresponding {\bfseries Fomenko-Mishchenko algebra} to be the subalgebra $A_\chi(\g)$ of $S(\g)$ generated by $\partial^i_\chi P_j$ for various $i,j$.
\begin{theorem}[\cite{FFT}, Theorems 3.11, 3.14]\label{t:freegensofbigalg}
Let $\chi\in\g^*$ be regular. Then
    \begin{itemize}
        \item [(i)] The algebra $A_\chi(\g)$ is the free polynomial algebra in generators $\partial^i_\chi P_j$, $j=1\ldots l$ and $i=0\ldots \deg P_j-1$.
        \item [(ii)] The algebra $\mc B_\chi(\g)$ is the {\bfseries quantization} of $A_\chi(\g)$, i.e. its associated graded ring via the canonical filtration of $U(\g)$ coincides with $A_\chi(\g)$.
    \end{itemize}
\end{theorem}

\begin{example}\label{e:imageofcasimir}
{\rm
    It is useful to compute the image of the Segal-Sugawara vector
$$
S_1:=\sum_i X_i[-1]X^i[-1]\in \mathfrak z(\widehat\g)
$$
in the Gaudin algebras, where $u$ is set to be $0$ for simplicity. Then, according to Proposition \ref{prop:genpropofgaudin}(i), 
\begin{align*}
    \Psi_{0;z,\infty}(S_1)&=\sum_i \Psi_{0;z,\infty}(X_i[-1])\Psi_{0;z,\infty}(X^i[-1])\\
    &=\sum_i(z^{-1}X_i\otimes 1+1\otimes X_i)(z^{-1}X^i\otimes 1+1\otimes X^i)\\
    &=z^{-2}\sum_iX_iX^i\otimes 1+z^{-1}\sum_i X_i\otimes X^i+1\otimes\sum_i X_iX^i.
\end{align*}
In particular, since $\mc Z_{z,\infty}^{1,2}(\g)$ is independent on $z$, it contains all the components
$$
\sum_iX_iX^i\otimes 1,\qquad\sum_i X_i\otimes X^i,\qquad1\otimes\sum_i X_iX^i.
$$

Note that evaluation of the middle element at $1\otimes\chi$ for $\chi\in\g^*$ gives the corresponding element $\chi\in\g$ via the chosen identification $\g\simeq \g^*$. Thus, we naturally obtain $\chi\in\mc B_\chi(\g)$.
This observation will be useful to us in Section \ref{s:transferatreps}.
}
\end{example}

\subsection{The center of the completed enveloping algebra}
It is instructive to introduce one more object for the completeness of our results. The reader may consult \cite[Section 5.2]{FFT} for details. As always, let $\kappa$ be the critical level, and define $U_\kappa(\widehat \g)$ to be the quotient of $U(\widehat\g)$ by the ideal $\lge\mathbf{1}-1\rge$. Its {\bfseries completion} is defined as
$$
\widetilde U_\kappa(\widehat\g):=\varprojlim U_\kappa(\widehat \g)/U_\kappa(\widehat \g)\cdot(t^N\g\jet t).
$$
We define $Z(\widehat \g)$ to be the center of $\widetilde U_\kappa(\widehat\g)$. Since $\mathbb V_0$ is a $\widetilde U_\kappa(\widehat\g)$-module, we have a natural map
$$
Z(\widehat \g)\to \operatorname{End}_{\widehat \g}\mathbb V_0=\mathfrak z(\widehat \g),
$$
which will be seen to be surjective in Theorem \ref{t:isoswithopers}.
%From now on and throughout the whole paper, all the constructions will be made in the critical level case. 
\section{Dynkin automorphisms}
Fix a semisimple algebraic group $G$ over $\C$ of rank $l$, its Borel subgroup $B$ and a maximal torus $H$. Denote by $N$ the unipotent radical $[B,B]$ of $B$ and by $\g,\b,\h,\n$ the corresponding Lie algebras. Any automorphism of the Dynkin diagram of $G$ giving an automorphism of $G$ preserving a pinning of $G$ is called a {\bfseries Dynkin automorphism}. Let $\sigma$ be a Dynkin automorphism of $G$ preserving $B$ and $H$. In particular, $\sigma$ preserves $N$ and the root datum $(X,X^\vee,\alpha_i,\alpha_i^\vee)$, where we mention simple roots only. 

Let $G_\sigma$ be the connected component of the identity of the fixed point subgroup of $G$. Similarly define $H_\sigma$, $B_\sigma$, $N_\sigma$. We sum up the properties of $G_\sigma$ in the following theorem, citing \cite[Section 2]{Ho}.

\begin{theorem}\label{t:dynkinfixeddatum}
The root datum of $G_\sigma$ has the form $(Y,Y^\vee,\alpha_\eta,\alpha_\eta^\vee)$, where 
\begin{itemize}
    \item $Y^\vee=(X^\vee)^\sigma$, the lattice of $\sigma$-invariants.
    \item $Y=X_\sigma$, the lattice of $\sigma$-coinvariants.
     \item The simple roots are parametrized by the $\sigma$-orbits $\eta$ and $\alpha_\eta$ is the image of any $\alpha_i\in X^\bullet$ for $i\in \eta$.
    \item The simple coroots are parametrized by the $\sigma$-orbits $\eta$ and
    $$
    \alpha_\eta^\vee=
    \begin{cases}
        2\sum_{i\in\eta}\alpha_i^\vee,&\eta=\{a,b\}, a-b\text{ in the Dynkin diagram},\\
        \sum_{i\in\eta}\alpha_\eta^\vee,&\text{otherwise}.
    \end{cases}
    $$
    \item The fundamental coweights are parametrized by the $\sigma$-orbits and $$
    \omega_\eta^\vee=\sum_{i\in\eta}\omega_i^\vee.
    $$
    \item The pairing between $Y$ and $Y^\vee$ is naturally induced from the pairing between $X$ and $X^\vee$.
\end{itemize}
Moreover, $(X^\vee_+)^\sigma=Y^\vee_+$.
\end{theorem}
\begin{proof}
    The only statement that is not mentioned in \cite{Ho} is the one about the fundamental coweights. It is clear from the description of simple roots that the elements $\sum_{i\in\eta}\omega_i^\vee$ form the dual basis for the simple roots, so we are done.
\end{proof}

\begin{corollary}\label{cor:invadjandsemistype}
    If $G$ is of adjoint type, then $G_\sigma$ is of adjoint type. If $G$ is simply-connected, then $G_\sigma$ is simply connected unless $G$ has type $A_{2n}$, in which case $G=\mathrm{SL}_{2n+1}$ and $G_\sigma=\mathrm{SO}_{2n+1}$.
\end{corollary}
\begin{proof}
    Assume that $G$ is of adjoint type. This implies that $X(G)=\Z\lge\alpha_i:i\rge$. But then 
    $$
    X(G_\sigma)=X(G)_\sigma=\Z\lge\alpha_i\rge_\sigma=\Z\lge\alpha_\eta:\eta\rge,
    $$
    which implies that $G_\sigma$ is of adjoint type.

    Assume now that $G$ is simply connected. Then $X^\vee(G)=\Z\lge\alpha_i^\vee\rge$. It follows from Theorem \ref{t:dynkinfixeddatum} that 
    $$
    [X^\vee(G_\sigma):\Z\lge\alpha_\eta^\vee:\eta\rge]=[\Z\lge\alpha_i^\vee:i\rge^\sigma:\Z\lge\alpha_\eta^\vee:\eta\rge]=[\Z\lge\textstyle\sum_{i\in\eta}\alpha_i^\vee:\eta\rge:\Z\lge\alpha_\eta^\vee:\eta\rge]=2^r,
    $$
    where $r$ is the number of orbits $\{a,b\}$ such that $a$ and $b$ are connected in the Dynkin diagram. Looking at all the Dynkin diagrams, it is easy to see that $r=0$ unless the Dynkin diagram has type $A_{2n}$, where $r=1$. This finishes the proof. 
\end{proof}

Using this corollary, we can easily list all the examples of $(G,G_\sigma)\leftrightarrow (G^\vee,G_\sigma^\vee)$, where $G$ is of simply connected or adjoint type. They are the following:

{\small
\begin{center}
\begin{tabular}{ |c|c|c|c|c| }
 \hline
 \multirow{2}{*}{Type}&\multicolumn{2}{|c|}{Simply connected} &\multicolumn{2}{|c|}{Adjoint type}\\
 \cline{2-5}
  & $(G,G_\sigma)$&$(G^\vee,G_\sigma^\vee)$& $(G,G_\sigma)$&$(G^\vee,G_\sigma^\vee)$\\
 \hline
 $A_{2n-1}$&$(\mathrm{SL}_{2n},\mathrm{Sp}_{2n})$&$(\mathrm{PGL}_{2n},\mathrm{SO}_{2n+1})$&$(\mathrm{PGL}_{2n},\mathrm{Sp}_{2n}/\{\pm I\})$&$(\mathrm{SL}_{2n},\mathrm{Spin}_{2n+1})$\\
 \hline
$A_{2n}$&$(\mathrm{SL}_{2n+1},\mathrm{SO}_{2n+1})$&$(\mathrm{PGL}_{2n+1},\mathrm{Sp}_{2n})$&$(\mathrm{PGL}_{2n+1},\mathrm{SO}_{2n+1})$&$(\mathrm{SL}_{2n+1},\mathrm{Sp}_{2n})$\\
 \hline
 $D_{n}$&$(\mathrm{Spin}_{2n},\mathrm{Spin}_{2n-1})$&$(\mathrm{SO}_{2n}/\{\pm I\},\mathrm{Sp}_{2n-2}/\{\pm I\})$&$(\mathrm{SO}_{2n}/\{\pm I\},\mathrm{SO}_{2n-1})$&$(\mathrm{Spin}_{2n},\mathrm{Sp}_{2n-2})$\\
 \hline
$D_4$&$(\mathrm{Spin}_{8},G_{2,\mathrm{sc}})$&$(\mathrm{SO}_{8}/\{\pm I\},G_{2,\mathrm{ad}})$&$(\mathrm{SO}_{8}/\{\pm I\},G_{2,\mathrm{ad}})$&$(\mathrm{Spin}_{8},G_{2,\mathrm{sc}})$\\
 \hline
 $E_6$&$(E_{6,\mathrm{sc}},F_{4,\mathrm{sc}})$&$(E_{6,\mathrm{ad}},F_{4,\mathrm{ad}})$&$(E_{6,\mathrm{ad}},F_{4,\mathrm{ad}})$&$(E_{6,\mathrm{sc}},F_{4,\mathrm{sc}})$\\
 \hline
\end{tabular}    
\end{center}
}

\medskip

There is one more interesting example in type $D$: $(\mathrm{SO}_{2n},\mathrm{SO}_{2n-1})^\vee=(\mathrm{SO}_{2n},\mathrm{Sp}_{2n-2})$.

One could ask whether the fixed point subgroup is connected. We present the corresponding result below, however will not use it in the paper.
\begin{theorem}\
    \begin{itemize}
        \item [(i)] If $G$ is simply-connected, then the fixed point subgroup $G^\sigma$ is connected.
        \item [(ii)] If $G$ is of adjoint type, then $G^\sigma$ is connected if and only if $G$ is of type $A_{2n}$, $E_6$, or $D_4$ with $\sigma$ of order $3$.
    \end{itemize}
\end{theorem}
\begin{proof}
    Follows from 8.2, 9.7, 9.8 and 9.9 in \cite{St}.
\end{proof}

The last statement of Theorem \ref{t:dynkinfixeddatum} implies that we have an embedding 
$$
\mathrm{Rep}\,G_\sigma^\vee\hookrightarrow\mathrm{Rep}\,G^\vee
$$
as categories, sending the irreducible representation of $G_\sigma^\vee$ of highest weight $\lambda$ to the irreducible representation of $G^\vee$ of highest weight $\lambda$.

\begin{remark}
{\rm
    It is not an embedding of tensor categories. Example \ref{e:notmonoidal} gives a counterexample to this. 
}
\end{remark}

Since $\sigma$ acts on $X_+$ and $X_+^\vee$, then it acts on the tensor categories $\operatorname{Rep}G$ and $\operatorname{Rep}G^\vee$ by sending the irreducible representation of highest weight $\lambda$ to the irreducible representation of highest weight $\sigma^{-1}(\lambda)$. By the Tannakian formalism, the action of $\sigma$ on $\operatorname{Rep}G^\vee$ induces a compatible automorphism of $G^\vee$, which, according to \cite[Lemma 4.1]{Ho}, turns out to be a Dynkin automorphism.

Let $\lambda$ be a $\sigma$-invariant weight of $G^\vee$. Let $V(\lambda)$ and $W(\lambda)$ be the corresponding highest weight representations of $G^\vee$ and $G_\sigma^\vee$, respectively. Then $\sigma$ acts on $V(\lambda)$ by the rule 
$$
\sigma(g\cdot v)=\sigma(g)\cdot \sigma(v)\qquad g\in G,\mbox{ }v\in V(\lambda).
$$
This action is conventionally induced by letting $\sigma$ send the highest weight vector to itself. There is a remarkable formula due to Jantzen, originally proved in \cite{Ja} (for a modern treatment and links to various proofs look in 
\cite[1.1.1]{HS}):
\begin{theorem}[Jantzen's twining formula]\label{t:Jantzen}
    Choose $\mu\in X(G_\sigma^\vee)$. Then
    $$
    \mathrm{tr}(\sigma|V_\mu(\lambda))=\dim W_\mu(\lambda),
    $$
    where $V_\mu(\lambda)$ and $W_\mu(\lambda)$ denote the corresponding weight spaces. In particular, we have the global identity
    $$
    \mathrm{tr}(\sigma|V(\lambda))=\dim W(\lambda).
    $$
\end{theorem}
For a general overview of this theorem and links to various proofs, the reader may consult \cite{Ho}.

Using the constructions from Section \ref{s:gaudinalg}, we can see that the action of $\sigma$ on $\g$ and $\g^\vee$ induces an action on all Gaudin algebras, in particular, $\mc B^\lambda(\g^\vee)$ for $\sigma$-invariant $\lambda$. 
\begin{conjecture}[Conjecture 4.1, \cite{Ha}]\label{conj:isoofbigalg}
    There exists an isomorphism 
    $$
    \mc B^\lambda(\g^\vee)_\sigma\simeq \mc B^\lambda(\g_\sigma^\vee),
    $$
    where $\mc B^\lambda(\g^\vee)_\sigma$ denotes the quotient of $\mc B^\lambda(\g^\vee)$ by the ideal generated by $x-\sigma(x)$ for all $x\in \mc B^\lambda(\g^\vee)$.
\end{conjecture}

Clearly, we are interested not in finding an arbitrary isomorphism, but in finding a "good" one, and part of the conjecture lies in finding the proper meaning of "good". This question motivates the whole paper.

But first of all, let us try to see how the action of $\sigma$ on the categories of representations looks in particular examples. We first describe a general pattern. Using the description of the fundamental coweights of $\g_\sigma$, note that if
$$
\lambda=\sum_{\eta}a_\eta\omega_\eta^\vee\in X_+(G_\sigma^\vee),
$$
then the corresponding element in $X_+(G^\vee)$ is 
$$
\lambda=\sum_{\eta}a_\eta\sum_{i\in \eta}\omega_i^\vee\in X_+(G^\vee).
$$

\subsection{Type $A_{2n-1}$}
Take the unique Dynkin automorphism of $\mathrm{SL}_{2n}$, which has the form
$$
\sigma(A)=-J(A^T)^{-1}J^{-1},
$$
where $J_{ij}=(-1)^{i}\delta_{i+j,2n+1}$. Then the fixed point subgroup is $\mathrm{Sp}_{2n}$ and the Langlands duals are $\mathrm{PGL}_{2n}$ and $\mathrm{SO}_{2n+1}$, respectively. However, we would like to use conventions from \cite[Sections 16,17]{FH}, hence we change the basis of $\C^{2n}$ so that $J$ is sent to
$$
J_0:=\begin{pmatrix}
    0&-I_n\\
    I_n&0
\end{pmatrix}.
$$
Note that the matrix of this basis change fixes the maximal torus of $\mathrm{SL}_{2n}$, hence the base change changes the lattice by an element of the Weyl group. Therefore, after we get a weight $\lambda$ of $\mathrm{PGL}_n$ corresponding to a dominant weight of $\mathrm{Sp}_{2n}$, the desired fundamental weight will be the dominant weight of $\mathrm{PGL}_n$ lying in the same Weyl orbit as $\lambda$.

Let's recall the conventions in \cite[Sections 16,17]{FH}. The maximal tori of $\mathrm{Sp}_{2n}$ and $\mathrm{SL}_{2n}$ are precisely the diagonal matrices lying in these groups. The coweights of $\mathrm{SL}_{2n}$ are $H_i:=E_{ii}$ and the coweights of $\mathrm{Sp}_{2n}$ are $H'_i:=E_{ii}-E_{n+i,n+i}=H_i-H_{n+i}$. The fundamental Weyl chambers in these cases are
\begin{align*}
    \Bigg\{\sum_{i=1}^{2n}a_iH_i&:\sum_i a_i=0,a_1\ge a_2\ge\ldots\ge a_{2n}\Bigg\},\\
    \Bigg\{\sum_{i=1}^{n}a_iH'_i&:a_1\ge a_2\ge\ldots\ge a_{n}\Bigg\}.
\end{align*}

Note that if we compose the map between the coweights with the element
$$
\begin{pmatrix}
    1&2&\cdots&n&n+1&n+2&\cdots&2n-1&2n\\
    1&2&\cdots&n&2n&2n-1&\cdots&n+2&n+1
\end{pmatrix}
$$
of the Weyl group of $\mathrm{PGL}_{2n}$, then we will get the map
$$
H'_i\mapsto H_{i}-H_{2n+1-i},
$$
and it is easy to see that it sends the dominant coweights of $\mathrm{Sp}_{2n}$ to dominant coweights of $\mathrm{SL}_{2n}$. Note that in terms of representations of $\mathrm{SL}_{2n}$, the map is

$$
(a_1,\ldots,a_n)\mapsto (a_1+a_1,\ldots,a_1+a_n, a_1-a_n,\ldots,a_1-a_2,0).
$$
This is how this map looks in terms of the associated Young diagrams:

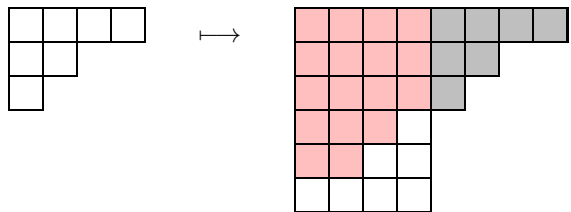
\begin{center}
    \ytableausetup
{boxsize=1.25em}
\ytableausetup{aligntableaux=top}
\ydiagram{4,2,1}\qquad$\longmapsto$\qquad
\ytableaushort
{\none}
* {8,6,5,4,4,4}
* [*(pink)]{4,4,4,3,2}
*[*(lightgray)]{4+4,4+2,4+1}
\end{center}
where the pink-white rectangle to the left of the grey diagram has size $a_1\times 2n$ and the colored diagram is the image of the map.

It follows from the above that the action of $\sigma$ on the coweights of $\mathrm{SL}_{2n}$ can be rewritten as
$$
(a_1,\ldots,a_{2n})\mapsto (-a_{2n},\ldots,-a_1),\qquad (1,\ldots,1)=0,
$$
which corresponds to dualizing a representation. Therefore, 
$$
V(\sigma(\lambda))=V(\lambda)^*
$$
as representations of $\mathrm{PGL}_{2n}$ and the irreducible $\sigma$-invariant representations are exactly the self-dual ones.

\begin{example}\label{e:notmonoidal}
{\rm
    Let $n=2$ and consider partitions $(2,2)$ and $(1,1)$ giving representations of $\mathrm{SO}_{5}$. They get mapped to $\lambda=(4,2,2)$ and $\mu=(2,2)$, respectively, giving representations of $\mathrm{PGL}_{4}$. It is easy to see using the Littlewood-Richardson rule that it contains the representation $(6,3,2,1)\sim (5,2,1)$:
    \begin{center}
    \ytableausetup
{boxsize=1.25em}
\ytableausetup{aligntableaux=top}
\ytableaushort
{\none\none\none\none 11,\none\none 2,\none,2}
* {6,3,2,1}
*[*(white) \bullet]{4,2,2}
\end{center}
This representation is not $\sigma$-invariant since its dual is $(-1,-2,-3,-6)\sim (5,4,3)$, hence $V(\lambda)\otimes V(\mu)$ does not come from a representation of $\mathrm{SO}_5$.
}
\end{example}

\subsection{Type $D_{n}$, $n\ge 5$}
We have a unique Dynkin involution here too. Consider the pair $\mathrm{SO}_{2n},\mathrm{SO}_{2n-1}$. The Langlands duals are $\mathrm{SO}_{2n},\mathrm{Sp}_{2n-2}$. If $\omega_1,\ldots, \omega_{n}$ are the fundamental coweights of $\mathfrak{so}_{2n}$, then $\omega_1,\ldots, \omega_{n-2},\omega_{n-1}+\omega_{n}$ are the fundamental coweights of $\mathfrak{so}_{2n-1}$. Following the conventions of \cite[Sections 16, 18]{FH}, we see that the fundamental weights of $\mathfrak{sp}_{2n-2}$ are $\omega'_i=L'_1+\ldots+L'_i$, and the fundamental weights of $\mathfrak{so}_{2n}$ are the same except
$$
\omega_{n-1}=\frac{L_1+\ldots+L_n}2,\qquad \omega_{n}=\frac{L_1+\ldots+L_{n-1}-L_n}2.
$$
We get from $\omega'_i=\omega_i$ for $i<n-1$ and $\omega'_{n-1}=\omega_{n-1}+\omega_n$ that $L'_i$ correspond to $L_i$ for $i<n$. In other words, the representation of $\mathrm{Sp}_{2n-2}$ with highest weight $a_1L'_1+\ldots+a_{n-1}L'_{n-1}$ is sent to the representation of $\mathrm{SO}_{2n}$ with highest weight $a_1L_1+\ldots+a_{n-1}L_{n-1}$.

\section{Spaces of opers and notation}\label{s:generaltransfer}
For a precise treatment of opers, the reader may consult \cite[Section 3]{BD1}, \cite[Section 1]{F1}, \cite[Section 4]{FFT}, or \cite[Part I]{FG}. We also recommend \cite{BD2} and \cite[Section 4]{F3} for an accessible exposition.

\subsection{Definition of opers}
Let $G$ be a simple algebraic group of adjoint type with a fixed Borel subgroup $B$ and a maximal torus $H$. Denote by $N$ the unipotent radical $[B,B]$ of $B$ and by $\g,\b,\h,\n$ the corresponding Lie algebras. As always, denote by $\alpha_i$ and $\alpha_i^\vee$ simple roots and simple coroots, respectively. There is a unique open $B$-orbit 
$$
\mathbf O\subset [\n,\n]^\perp/\b\simeq \bigoplus_i\g_{-\alpha_i},
$$
which has the following properties:
\begin{itemize}
    \item It it invariant under the action of $N$ and $B/N\simeq H$ acts simply-transitively on it.
    \item $\mathbf O=\{\sum_{i}a_if_i:\forall_i\,a_i\ne 0\}$, where $f_i$ are generators of negative simple root spaces $\g_{\alpha_i}$.
\end{itemize}

From now and till the rest of the paper, we let $X$ be either a smooth connected curve over $\C$, the formal disc $D:=\spec \C\jet t$ or the punctured formal disc $D^\times:=\spec \C\rjet t$. Let $\Omega_X$ denote the canonical line bundle on $X$.

\begin{definition}
    Let $\mc F_B$ be a principal $B$-bundle on $X$ and $\mc F_G:=\mc F_B\times^BG$ its extended principal $G$-bundle on $X$. For a connection $\nabla$ on $\mc F_G$ we define its {\bfseries relative position with respect to $\mc F_B$} to be the section $\nabla/\mc F_B\in\Gamma((\g/\b)_{\mc F_B}\otimes \Omega_X)$ constructed as follows:
\begin{enumerate}
    \item Choose a connection $\nabla_B$ on $\mc F_B$ and lift it to $\nabla_G:=\nabla_B\times^BG$ on $\mc F_G$.
    \item Take the difference $\nabla-\nabla_G\in \Gamma(\g_{\mc F_G}\otimes \Omega_X)$ and define $\nabla/\mc F_B$ to be its image in $\Gamma((\g/\b)_{\mc F_B}\otimes \Omega_X)$ via the quotient map $\g\to\g/\b$.
\end{enumerate}
\end{definition}

We have to show that $\nabla/\mc F_B$ is well-defined. Suppose that we have chosen a different connection $\nabla_B'$. Then the definitions of $\nabla/\mc F_B$ differ by the image of $\nabla_G-\nabla'_G$ in $\Gamma((\g/\b)_{\mc F_B}\otimes \Omega_X)$. Since both $\nabla_G$ and $\nabla_G'$ are extended from $\mc F_B$, then $\nabla_G-\nabla'_G$ actually lies in $\Gamma(\b_{\mc F_B}\otimes \Omega_X)$. The split exact sequence
$$
0\to \b\to \g\to \g/\b\to 0
$$
gives rise to the exact sequence
$$
0\to \Gamma(\b_{\mc F_B}\otimes \Omega_X)\to \Gamma(\g_{\mc F_B}\otimes \Omega_X)\to \Gamma((\g/\b)_{\mc F_B}\otimes \Omega_X),
$$
which shows that the image of $\nabla_G-\nabla'_G$ in $\Gamma((\g/\b)_{\mc F_B}\otimes \Omega_X)$ is $0$, as desired.

%, and we will denote the associated $\C^\times$-principal bundle in the same way (it will always be clear from the context which definition we use).
\begin{definition}\label{def:opers}
    A {\bfseries $G$-oper on $X$} is a triple $(\mc F_G,\mc F_B,\nabla)$, where
    \begin{itemize}
        \item $\mc F_B$ is a principal $B$-bundle on $X$,
        \item $\mc F_G=\mc F_B\times^BG$ is the associated principal $G$-bundle on $X$,
        \item $\nabla$ is a connection on $\mc F_G$ such that the values of $\nabla/\mc F_G$ lie in $\mathbf O_{\mc F_B}$. In other words,
        $$
        \nabla/\mc F_G\in \Gamma(\mathbf O_{\mc F_B}\otimes \Omega_X)\subset \Gamma((\g/\b)_{\mc F_B}\otimes \Omega_X).
        $$
    \end{itemize}
    A {\bfseries $\g$-oper on $X$} is a $G$-oper on $X$ for the adjoint form $G$ of $\g$.
\end{definition}

The last condition should be thought of as a certain transversality condition, meaning that $\nabla$ is as far from a connection induced from a connection on $\mc F_B$ as possible. For instance, is easy to see from the arguments above that $\nabla/\mc F_B=0$ if and only if $\nabla$ is induced from a connection on $\mc F_B$. 

Before going further, let us show a local description of opers. Choose a closed point $x\in X$, an {\'e}tale coordinate $t$ around $x$ and a trivialization of $\mc F_B$ around $x$. Then the condition on $\nabla$ is equivalent to $\nabla$ having a form
\begin{equation}\label{eq:localexprforopers}
    \nabla=\nabla^0+\sum_if_i\phi_i(t)dt+\mathbf v(t)dt,
\end{equation}
where $\nabla^0$ is the flat connection on $\mc F_B$ given by the trivialization, $\phi_i\in \C\jet t$ don't vanish at $0$ and $\mathbf v(t)\in b\jet t=\b\otimes \C\jet t$. This definition depends on the trivialization of $\mc F_B$, and we would like to identify the same objects corresponding to different trivializations. Any two trivializations of $\mc F_B$ differ by an element $g\in B\jet t$, hence $B$ acts on (\ref{eq:localexprforopers}). The action is called the {\bfseries gauge} action and takes the form
$$
g\cdot \left(\nabla^0+\sum_if_i\phi_i(t)dt+\mathbf v(t)dt\right)=\nabla^0+\sum_i\mathrm{Ad}_g(f_i)\phi_i(t)dt+\mathrm{Ad}_g(\mathbf v(t))dt-g^{-1}\cdot d(g).
$$
Therefore, the set of opers can be identified with
$$
\frac{\left\{\nabla^0+\sum_{i=1}^l\psi_if_i+\mathbf v(t):\mathbf v(t)\in\b\jet t,\psi_i(0)\ne 0\right\}}{B\jet t}.
$$

We can extend Definition \ref{def:opers} to the $R$-family of opers for any $\C$-algebra $R$, requiring that all the constructions are made over $X\times_{\spec\C}\spec R$ and use $G\times_{\spec\C}\spec R$. Note that all principal bundles are required to be locally trivial in the {\'e}tale topology, so in particular, $\mc F_B$ are trivial in opers on $D$. This way, we define a functor on $\operatorname{AffSch}^{\mathrm{op}}_\C$, which we denote by $\op_G(X)$. Similarly, we can write the local description of opers over any $\C$-algebra $R$. In this survey, we will give the descriptions of all our constructions for $\C$-points only, and the precise treatment of $R$-points can be found in \cite[1.1-1.5]{FG}. Before we provide a theorem that guarantees that this space has a nice structure in the cases interesting to us, we will give several new definitions.

\subsection{Opers with singularities}
\begin{definition}
    Let $x\in X$ be a closed point. A {\bfseries $G$-oper in $X$ with singularity of order $k$ at $x$} is a triple $(\mc F_G,\mc F_B,\nabla)$, where
    \begin{itemize}
        \item $\mc F_B$ is a principal $B$-bundle on $X$,
        \item $\mc F_G=\mc F_B\times^BG$ is the associated principal $G$-bundle on $X$,
        \item $\nabla$ is a connection on $\mc F_G$ with pole of order $k$ at $x$ such that
        $$
        \left[\nabla-\nabla_G\mbox{ }\mathrm{mod}\mbox{ }\b_{\mc F_B}\otimes \Omega_X(k\cdot x)\right]\in \Gamma(\mathbf O_{\mc F_B}\otimes \Omega_X(k\cdot x))\subset \Gamma((\g/\b)_{\mc F_B}\otimes \Omega_X(k\cdot x)),
        $$
    \end{itemize}
        where $\nabla_G$ is some connection on $\mc F_G$ extended from a connection on $\mc F_B$ and $\Omega_X(k\cdot x)$ denotes the tensor product of $\Omega_X$ with the line bundle on $X$ corresponding to the divisor $k\cdot x$. The corresponding functor is denoted by $\op_{G,x}^{\mathrm{ord}_k}(X)$. 
        
        Similarly, for any divisor $C$ on $X$ we can define {\bfseries opers with singularities at $C$} and denote this space by $\op_{G,C}(X)$.
\end{definition}

In terms of a local trivialization at $x$, such connections take the form
\begin{equation}\label{eq:localoperswithsing}
    \nabla=\nabla^0+t^{-k}\left(\sum_if_i\phi_i(t)dt+\mathbf v(t)dt\right),
\end{equation}
where $\phi_i(t)\in \C\jet t$ don't vanish at $0$. We can write
\begin{align*}
    \op_{G,x}^{\mathrm{ord}_k}(D)&\simeq \frac{\left\{\nabla^0+t^{-k}\left(\sum_{i=1}^l\psi_if_i+\mathbf v(t)\right):\mathbf v(t)\in\b\jet t, \psi_i(0)\ne 0\right\}}{B\jet t}\\
    &\simeq \frac{\left\{\nabla^0+t^{-k}\left(\sum_{i=1}^lf_i+\mathbf v(t)\right):\mathbf v(t)\in\b\jet t\right\}}{N\jet t}.
\end{align*}

Let $(T_x^\times)^{\otimes (m-1)}$ denote the $\C^\times$-bundle over $x$ corresponding to $(T_xX)^{\otimes (m-1)}$. Following \cite[4.3]{FFT}, we define the $m$-residue map
$$
\operatorname{Res}^m_x:\op_G^{\mathrm{ord}_m}(X)\to (\h\git W)\times^{\C^\times}(T_x^\times)^{\otimes (m-1)}
$$
sending the connection
$$
\nabla^0+t^{-k}\left(\sum_{i=1}^lf_i+\mathbf v(t)\right)
$$
to the projection of $\sum_if_i+\mathbf v(0)\in\g$ to $\g\git G\simeq \h\git W$. The twist by $T_x$ is introduced to make the map independent of the choice of coordinate $t$.

We have the following results about the spaces constructed:
\begin{theorem}[3.1.11 and 3.8.23, \cite{BD1}]\
    \begin{itemize}
        \item[(i)] If $X$ is a smooth complete curve over $\C$ or $X=D$, $C$ is a divisor on $X$, then $\mathrm{Op}_{\g,C}(X)$ is representable by an affine scheme.
        \item[(ii)] If $X$ is any smooth curve or $X=D^\times$, then $\mathrm{Op}_\g(X)$ is representable by an ind-scheme.
    \end{itemize}
\end{theorem}

It appears that the condition on $\nabla$ in the definition of opers is tremendously restrictive.
\begin{theorem}
Let $G$ be of adjoint type. Then
    \begin{itemize}
        \item[(i)] For any $\g$-oper $(\mc F_G,\mc F_B,\nabla)$ on $X$ the bundle $\mc F_B$ is isomorphic to a certain canonical $B$-bundle $\mc F_B^0$, which doesn't depend on the oper.
        \item[(ii)] Let $X$ be a complete curve, $x\in X$ and consider the map $\op_\g(X)\to\op_\g(D_x)$ restricting an oper on $X$ to the formal disc at $x$. Then this map is a closed immersion.
    \end{itemize}
\end{theorem}
\begin{proof}
    Item $(i)$ follows from \cite[3.1.10]{BD1}. Item $(ii)$ follows from \cite[3.3.3]{BD1}. 
\end{proof}
\begin{remark}
   {\rm
    It is helpful to know how to construct this bundle $\mc F_B^0$. Assume that we know such a bundle $\mc F^0$ for $\mathfrak{sl_2}$-opers. Let $\mathfrak{s}$ be a principal $\mathfrak{sl}_2$-subalgebra of $\g$ and $S$ be the corresponding $\mathrm{PGL}_2$-subgroup of the adjoint form $G$ of $\g$ with the Borel subgroup $B_2=B\cap \mathrm{PGL}_2$. Then for any $\mathfrak{sl_2}$-oper $(\mc F^0,\nabla)$ (we omitted the $S$-bundle for simplicity) the extended pair $(\mc F^0\times^{B_2}B,\nabla\times^SG)$ is a $G$-oper on $X$. Therefore, $\mc F_B^0\simeq \mc F^0\times^{S}G$. The precise construction of $\mc F^0$ is described in \cite[3.3.10]{BD1}.
   } 
\end{remark}

\subsection{Canonical representatives}\label{s:canrepr}
Choose a set of generators $\{f_i\}$, $i=1,\ldots,l$ of the negative simple root spaces of $\g$. Let $\rho^\vee\in\h$ be the sum of the fundamental coweights of $\g$ and $p_{-1}:=\sum_{i=1}^lf_i$. With respect to the grading on $\n$ defined by $\mathrm{ad}\,\rho^\vee$ (which in addition coincides with the height gradation), there exists a unique element $p_1$ of degree $1$ in $\n$ such that $\{p_{-1},2\rho^\vee,p_1\}$ is an $\mathfrak{sl}_2$ triple. 

Then we define $V_{\mathrm{can}}(\g)=V_{\mathrm{can}}\subset \n$ to be the $p_1$-invariant elements of $\n$.

\begin{theorem}[See \cite{FFT}, Sections 4.1-4.2]
    Let $G$ be of adjoint type. If the bundle $\mc F_B$ in $\g$-opers over $X$ is trivializable, then we have the isomorphisms:
\begin{align}
    V_{\mathrm{can}}\jet t&\simeq \left\{\nabla^0+\sum_{i=1}^lf_i+\mathbf v(t):\mathbf v(t)\in V_{\mathrm{can}}\jet t\right\}\label{isowithvcan}\\\notag
    &\xrightarrow{\sim} \frac{\left\{\nabla^0+\sum_{i=1}^lf_i+\mathbf v(t):\mathbf v(t)\in\b\jet t\right\}}{N\jet t}\\\notag
    &\xrightarrow{\sim} \frac{\left\{\nabla^0+\sum_{i=1}^l\psi_if_i+\mathbf v(t):\mathbf v(t)\in\b\jet t,\psi_i(0)\ne 0\right\}}{B\jet t}\\\notag
    &\xrightarrow{\sim} \mathrm{Op}_G(X),
\end{align}
where the maps are induced by inclusions of sets.
\end{theorem}

\subsection{Notation}
We will denote the space of $G$-opers on a space $X$ by $\mathrm{Op}_G(X)$. We will sometimes speak of $\g$-opers $\op_\g(X)$ as $G$-opers for the adjoint form $G$ of $\g$. For $z_1,\ldots,z_n\in\mathbb P^1$, $\chi\in \g^*$ and nonnegative integers $m_i$ define the following spaces:
\begin{itemize}
    \item $\mathrm{Op}_G(\mathbb P^1)_{(z_i),\infty}^{(m_i),m_\infty}$ is the space of opers on $\mathbb P^1$ with singularities $m_i$ at the points $z_i$.
    \item $\mathrm{Op}_G(\mathbb P^1)_{(z_i),\infty}$ is the space of opers on $\mathbb P^1$ with singularities of finite orders at the points $z_i$ and $\infty$.
    \item $\mathrm{Op}_{G}(\mathbb P^1)_{0,\infty}^{1,2,\lambda}$ is the subspace of $\mathrm{Op}_{G}(\mathbb P^1)_{0,\infty}^{1,2}$ consisting of opers with $1$-residue $\pi(-\lambda-\rho)$ at $0$, where $\rho$ is the Weyl vector.
    \item $\mathrm{Op}_{G}^{\mathrm{MF}}(\mathbb P^1)_{0,\infty}^{1,2,\lambda}$ is the subspace of $\mathrm{Op}_{G}(\mathbb P^1)_{0,\infty}^{1,2,\lambda}$ consisting of opers with trivial monodromy at $0$.
    \item $\mathrm{Op}_{G}(\mathbb P^1)_{0,\infty,\pi(-\chi)}^{1,2}$ is the subspace of $\mathrm{Op}_{G}(\mathbb P^1)_{0,\infty}^{1,2}$ consisting of opers with $2$-residue $\pi(-\chi)$ at $\infty$, where $\pi:\h\to\h\git W$ is the canonical projection.
    \item $\mathrm{Op}_{G}^{\mathrm{MF}}(\mathbb P^1)_{0,\infty,\pi(-\chi)}^{1,2}$ is the subspace of $\mathrm{Op}_{G}(\mathbb P^1)_{0,\infty,\pi(-\chi)}^{1,2}$ consisting of opers with trivial monodromy at $0$.
    \item $\mathrm{Op}_{G}(\mathbb P^1)_{0,\infty,\pi(-\chi)}^{1,2,\lambda}=\mathrm{Op}_{G}(\mathbb P^1)_{0,\infty}^{1,2,\lambda}\cap \mathrm{Op}_{G}(\mathbb P^1)_{0,\infty,\pi(-\chi)}^{1,2}$.
    \item $\mathrm{Op}_{G}^{\mathrm{MF}}(\mathbb P^1)_{0,\infty,\pi(-\chi)}^{1,2,\lambda}=\mathrm{Op}_{G}^{\mathrm{MF}}(\mathbb P^1)_{0,\infty}^{1,2,\lambda}\cap \mathrm{Op}_{G}(\mathbb P^1)_{0,\infty,\pi(-\chi)}^{1,2}$.
\end{itemize}

\subsection{Connection between opers and Gaudin algebras} In this subsection, we state the following theorem that will play the key role in proving our main result.
\begin{theorem}\label{t:isoswithopers}
    For generic $\chi$ we have the following commutative diagrams:
\begin{center}
    \begin{tikzcd}
Z_u(\widehat{\g_\kappa^\vee}) \arrow[d, two heads]                            & \mathrm{Fun}\,\mathrm{Op}_\g(D_u^\times)\arrow[l, "\sim"'] \arrow[d, two heads]                                     & 
\mathfrak z_u(\widehat{\g^\vee}) \arrow[d, two heads]                        & \mathrm{Fun}\,\mathrm{Op}_\g(D_u)\arrow[d, two heads] \arrow[l, "\sim"']\\
\mathfrak z_u(\widehat{\g^\vee}) \arrow[d, two heads]                        & \mathrm{Fun}\,\mathrm{Op}_\g(D_u)\arrow[d, two heads] \arrow[l, "\sim"']                                           &{\mathcal B(\g^\vee)} \arrow[d, two heads] & {\mathrm{Fun}\,\mathrm{Op}_\g(\mathbb P^1)_{0,\infty}^{1,2}} \arrow[l, "\sim"'] \arrow[d, two heads]                \\
{\mathcal Z_{(z_i),\infty}(\g^\vee)}\arrow[d, two heads] & \mathrm{Fun}\,\mathrm{Op}_\g(\mathbb P^1)_{(z_i),\infty} \arrow[l, "\sim"']\arrow[d, two heads]                                           &{\mathcal B_\chi(\g^\vee)} \arrow[d, two heads]               & {\mathrm{Fun}\,\mathrm{Op}_\g(\mathbb P^1)_{0,\infty,\pi(-\chi)}^{1,2}} \arrow[l, "\sim"']\arrow[d, two heads]\\
{\mathcal Z_{(z_i),\infty}^{(m_i),m_\infty}(\g^\vee)} & {\mathrm{Fun}\,\mathrm{Op}_\g(\mathbb P^1)_{(z_i),\infty}^{(m_i),m_\infty}} \arrow[l, "\sim"']&{\mathcal B_\chi^\lambda(\g^\vee)}                & {\mathrm{Fun}\,\mathrm{Op}_\g^{\mathrm{MF}}(\mathbb P^1)_{0,\infty,\pi(-\chi)}^{1,2,\lambda}} \arrow[l, "\sim"']\makebox[0pt][l]{\,,}
\end{tikzcd}
\end{center}       
where the map $\mathfrak z_u(\widehat{\g^\vee})\to \mathcal Z_{(z_i),\infty}(\g^\vee)$ is the map $\Psi_{u,(z_i),\infty}$ and all other vertical maps are the natural surjections.
\end{theorem}
\begin{proof}
    The existence of the diagram on the left follows from \cite[Proposition 5.3, Theorem 5.4, Theorem 5.7]{FFT}. Similarly, the diagram on the right is given by \cite[Theorem 5.8]{FFT} and \cite[Theorem A]{FFR}.
\end{proof}

\section{Actions of Dynkin automorphisms on spaces of opers}\label{sec:dynkinonopers}
Let $X$ be a smooth curve. The automorphism $\sigma$ defines an automorphism of $\mathrm{Op}_G(X)$, which can be described in one of the following equivalent ways:
    
    \begin{enumerate}
        \item Realizing the elements of $\mathrm{Op}_G(X)$ as triples $(\mc F_G, \mc F_B, \nabla)$, we define
$$
\sigma(\mc F_G, \mc F_B, \nabla)=(\mc F_G\times^{G}G(\sigma), \mc F_B\times^{B}B(\sigma), \nabla\times^{G}G(\sigma)),
$$      
where $G(\sigma)$ stands for the $G$-torsor $G$ with the left action
    $$
    g\cdot h:=\sigma(g)h
    $$
    and the standard right action by multiplication.
    \item Fixing isomorphisms
    $$
    (\mc F_G,\mc F_B)\simeq (\mc F_G\times^{G}G(\sigma),\mc F_B\times^{B}B(\sigma)),
    $$ 
    we can identify these triples and define $\sigma$ as acting on the connection. The action is induced from the action of $\sigma$ on $\g$ in $\Gamma(\Omega_{X}\otimes \g_{\mc F_G})$ after we have chosen a $\sigma$-invariant connection $\nabla^0$, where $\g_{\mc F_G}:=\mc F_G\times^G\g$.
    \item In terms of a local trivialization of $X$ with a local coordinate $t$,
$$
\sigma\left(\nabla^0+\sum_{i=1}^l\psi_if_i+\mathbf v(t)\right)=\nabla^0+\sum_{i=1}^l\psi_i\sigma^{-1}(f_i)+\sigma^{-1}(\mathbf v(t)).
$$
    \end{enumerate}

Now, we would like to construct a canonical isomorphism between $\op_\g(X)^\sigma$ and $\op_{G_\sigma}(X)$. We start with the following motivation in the case when $\mc F_B$ is trivializable over $X$. We use the notation from Section \ref{s:canrepr}. Choose a set of generators $\{f_i\}$, $i=1,\ldots,l$ of the negative simple root spaces of $\g$ on which $\sigma$ acts by permutation.

We see from Theorem \ref{t:dynkinfixeddatum} that if $\omega_1^\vee,\ldots,\omega_l^\vee$ is the set of fundamental coweights of $\g$, i.e. elements forming the dual basis to $\alpha_i$, then 
$$
\omega_\eta^\vee:=\sum_{i\in\eta}\omega_i^\vee,\qquad \text{$\eta$ is a $\sigma$-orbit in $\{1,...,l\}$},
$$
form the complete set of fundamental coweights of $\g_\sigma$, which implies that $\rho^\vee$ is fixed by $\sigma$ and equals the sum of fundamental coweights of $\g_\sigma$, i.e. coincides with the Weyl covector for $\g_\sigma$. Similarly, 
$$
f_\eta:=\sum_{i\in\eta}f_i
$$
are the generators of negative simple root spaces of $\g_\sigma$, hence 
$$
p_{-1}=\sum_{\eta}f_\eta,
$$
so the same statement is true for $p_{-1}$. Thus, the element $p_1$ defined for $\g$ is fixed by $\sigma$ and coincides with the analogous element $p_1$ defined for $\g_\sigma$.
Therefore, $\sigma$ acts on $V_{\mathrm{can}}(\g)$ and
$$
V_{\mathrm{can}}(\g)^\sigma=(\n^{p_1})^\sigma=(\n^\sigma)^{p_1}=V_{\mathrm{can}}(\g_\sigma).
$$

Moreover, it is easy to see that the action of $\sigma$ on $V_{\mathrm{can}}$ commutes with the isomorphisms (\ref{isowithvcan}). When $G$ is of adjoint type and $\mc F_B$ is trivializable for all opers $(\mc F_G,\mc F_B,\nabla)$ on $X$, this gives rise to the commutative diagram

\begin{center}
    \begin{tikzcd}
V_{\mathrm{can}}(\g)^\sigma\jet t \arrow[d, "\simeq"] & V_{\mathrm{can}}(\g_\sigma)\jet t \arrow[l, "\sim"'] \arrow[d, "\simeq"] \\
\mathrm{Op}_\g(X)^\sigma                               & \mathrm{Op}_{\g_\sigma}(X) \arrow[l, "\sim"']   \makebox[0pt][l]{\,,}                         
\end{tikzcd}
\end{center}
where the map below can be alternatively described as extending the triple $(\mc F_{G_\sigma},\mc F_{B^\sigma},\nabla)$ by $G$ and $B$.

\begin{theorem}\label{t:sigma-invariantopers}
    Let $X$ be a smooth curve or one of the discs $D$ or $D^\times$. Then the map
    $$
    \mathrm{Op}_{\g_\sigma}(X)\to \mathrm{Op}_\g(X)
    $$
    is an isomorphism onto $\mathrm{Op}_\g(X)^\sigma$.

    Moreover, a similar map
    $$
    \mathrm{Op}_{\g_\sigma}(X)_{\mathrm{sing}}\to \mathrm{Op}_\g(X)_{\mathrm{sing}}^\sigma
    $$
    of opers with singularities at finitely many points is an isomorphism. It preserves the orders of singularities and residues, as well as the monodromy-free property, hence gives an isomorphism for the subspaces of opers with prescribed poles and residues at prescribed points.
\end{theorem}
\begin{proof}
Let $G$ be the adjoint form of $\g$. It follows from \cite[3.1.10]{BD1} that the bundles $\mc F_B$ are the same for all $G$-opers on $X$, which implies that it is sufficient to check the isomorphism at the level of connections. Injectivity follows from the fact that $G$-opers for $G$ of adjoint type have no nontrivial automorphisms (\cite[3.1.4]{BD1}). For surjectivity, pick a $\sigma$-invariant $G$-oper $(\mc F_G,\mc F_B,\nabla)$ on $X$. Then we can think of the $\sigma$-action on it as the action on the connection. So, we have that $\sigma(\nabla)=\nabla$. Choose any connection $\nabla'$ on $\mc F_{G}$ extended from a connection on $\mc F_{G_\sigma}$. Then it is $\sigma$-invariant, hence $\nabla-\nabla'\in \Gamma(\Omega_X\otimes\g_{\mc F_G})$ is $\sigma$-invariant as well. Then
  $$
\nabla-\nabla'\in\Gamma(\Omega_X\otimes\g_{\mc F_G})^\sigma=\Gamma(\Omega_X\otimes\g_{\mc F_{ B_\sigma}})^\sigma=\Gamma(\Omega_X\otimes(\g_\sigma)_{\mc F_{ B_\sigma}})=\Gamma(\Omega_X\otimes(\g_\sigma)_{\mc F_{G_\sigma}}),
    $$
which implies that $\nabla$ is extended from $G_\sigma$. 

Now, we need to check the extra condition on $\nabla/\mc F_B$. Let $\mathbf O=\mathbf O(G):=B\cdot p_{-1}\subset \g/\b$ be the open $B$-orbit in $\g/\b$. Then the $B$-action on it factors through $B/N=H$, which acts simply transitively on it. Then 
$$
\mathbf O(G)^\sigma=\left\{\sum_ia_if_i:a_i=a_{\sigma^{-1}(i)}\right\}=\left\{\sum_\eta a_\eta f_\eta\right\}=\mathbf O(G_\sigma),
$$
where $\eta$ runs over the set of $\sigma$-orbits in the Dynkin diagram of $\g$. Thus,
    $$
\nabla/\mc F_B\in\Gamma(\Omega_X\otimes\mathbf O_{\mc F_G})^\sigma=\Gamma(\Omega_X\otimes\mathbf O_{\mc F_{ B^\sigma}})^\sigma=\Gamma(\Omega_X\otimes(\mathbf O^\sigma)_{\mc F_{ B^\sigma}})=\Gamma(\Omega_X\otimes\mathbf O(G_\sigma)_{\mc F_{ B^\sigma}}),
    $$
which implies that the triple is extended from an oper on $G_\sigma$, as desired.

The proof of the isomorphism from the second part is similar. The last statement is clear from the local description of opers with singularities (\ref{eq:localoperswithsing}) and monodromy-free opers from \cite[4.4]{FFT}.
\end{proof}

We end this section with the following important result, which follows from our discussion concerning the triple $\{p_{-1}, 2\rho^\vee, p_1\}$:
\begin{proposition}\label{p:quotspacesisosfordynkin}
    The restriction map
    \begin{equation}
\spec S(\g^*)^\g_\sigma\to\spec S(\g_\sigma^*)^{\g_\sigma}\label{eq:restrmapforsigmainv}
    \end{equation}
    is an isomorphism. In particular, 
    $$
    (\g\git G)^\sigma=(\h\git W)^\sigma=\h_\sigma\git W^\sigma=\g_\sigma\git G_\sigma.
    $$
\end{proposition}
\begin{proof}
    Due to \cite[Theorem 0.10]{Ko}, the Kostant section $p_{-1}+\g^{p_1}$ intersects each regular $G$-orbit in $\g$ in exactly one point and the restriction map $S(\g^*)^\g\to \C[p_{-1}+\g^{p_1}]$ is an isomorphism. Then all the statements follow from the equality
    $$
    (p_{-1}+\g^{p_1})^\sigma=p_{-1}+\g_\sigma^{p_1}.
    $$
\end{proof}
\begin{remark}
{\rm
    Note that this proposition doesn't hold for a general $G$-automorphism $\sigma$. A discussion on the surjectivity of the map \ref{eq:restrmapforsigmainv} is presented in \cite{PY}, Remark 4.6 in particular.
}
\end{remark}
\section{Coinvariants of generalized Gaudin algebras}
Our goal is to establish a canonical isomorphism between $\mathfrak z(\widehat{\g^\vee})_\sigma$ and $\mathfrak z(\widehat{\g_\sigma^\vee})$, where $\g_\sigma$ is the subspace of $\sigma$-invariant elements of $\g$. For this, we will use the isomorphism of $\mathfrak z$ with the space of opers, but first we need to define an action of $\sigma$ on $\mathfrak z(\widehat{\g^\vee})$ compatible with this isomorphism and the action of $\sigma$ on the space of opers.

We can define a $\sigma$-action on $\mathfrak z(\widehat\g)$ in the following ways:
\begin{enumerate}
        \item Consider the inclusion $\mathfrak z(\widehat\g)\subset U(\widehat\g_-)$. The action of $\sigma$ on $\g$ induces the action on $\widehat\g_-$, hence on $U(\widehat\g_-)$. As $\mathfrak z(\widehat\g)=U(\widehat{\g}_-)^{\g\jet t}$, the action restricts to $\mathfrak z(\widehat\g)$.
        \item Consider the interpretation $\mathfrak z(\widehat\g)=\mathrm{End}_{\widehat{\g}}\mathbb V_{0}$. Then $\sigma$ acts naturally on the right-hand side, which gives its action on $\mathfrak z(\widehat\g)$.
        \item We have an isomorphism
        $$
        \mathrm{Spec}\,\mathfrak z(\widehat\g)\simeq \mathrm{Op}_{\g^\vee}(D),
        $$
        and $\sigma$ acts on the right-hand side by its action on $\g^\vee$.
    \end{enumerate}

Before proving that all these constructions coincide, we present some useful definitions. Let $G$ be a simple and simply connected algebraic group. According to \cite[Section 2.4]{Zhu}, there is a canonical choice of a line bundle $\mc O(1)$ on the affine Grassmannian $\mathrm{Gr}_G=G\rjet t/G\jet t$ such that the Picard group $\operatorname{Pic}(\mathrm{Gr}_G)$ is isomorphic to $\Z\mc O(1)$. The {\bfseries critical level line bundle} is defined as $\mc L:=\mc O(-h^\vee)$, where $h^\vee$ is the dual Coxeter number of $G$.
\begin{lemma}\label{l:linebundlesareinvariant}
    Any automorphism of $\mathrm{Gr}_G$ acts on $\operatorname{Pic}(\mathrm{Gr}_G)$ trivially, where the action is via pullback.
\end{lemma}
\begin{proof}
    Let $\tau$ be an automorphism of $\mathrm{Gr}_G$. Then it gives an automorphism of $\operatorname{Pic}(\mathrm{Gr}_G)\simeq \Z\mc O(1)$, which implies that $\tau^*\mc O(m)$ either equals $\mc O(m)$ or $\mc O(-m)$ for any $m\in\Z$. Note that $\tau$ gives an isomorphism between global sections:
    $$
    \Gamma(\mathrm{Gr}_G,\mc O(m))\simeq \Gamma(\mathrm{Gr}_G,\tau^*\mc O(m)).
    $$
    It follows from \cite[Theorem 2.5.5]{Zhu} that 
    $$
    \dim \Gamma(\mathrm{Gr}_G,\mc O(m))>1\qquad\mbox{ } m\ge 1, 
    $$
    which implies that the dual bundles $\mc O(-m)$ for $m\ge 1$ have no global sections. Thus, $\tau^*\mc O(m)\ne \mc O(-m)$ for $m\ne 0$, which finishes the proof.
\end{proof}
\begin{theorem}\label{t:actionsonffcenter}
    All the actions $(1)-(3)$ listed above coincide.
\end{theorem}
\begin{proof}
 The equivalence $(1)\Leftrightarrow (2)$ is clear. We will show the equivalence $(1)\Leftrightarrow (3)$. Let $G^\vee$ be the adjoint form of $\g^\vee$. The map 
$$
    \mathrm{Spec}\,\mathfrak z(\widehat\g)\to \mathrm{Op}_{\g^\vee}(D)
$$
is the same as an element of $\mathrm{Op}_{G^\vee}(D)(\mathfrak z(\widehat\g))$, which is given by a $G^\vee$-bundle, its reduction to $B^\vee$ and a connection on $G^\vee$ satisfying certain properties. According to \cite[3.3.8]{BD1} and \cite[Proposition 5.1]{Ra}, it is sufficient to give here a $G^\vee$-bundle and its reduction to $B^\vee$ satisfying certain properties. Thus, our goal is to show that these $G^\vee$- and $B^\vee$-bundles are $\sigma$-invariant, where the $\sigma$-action on a bundle $P$ is given by
$$
P\mapsto (\sigma^{-1})^*(B\times^{G^\vee}G^\vee(\sigma)).
$$

By Tannakian formalism, a $G^\vee$-bundle can be equivalently described by a symmetric monoidal functor
$$
F:\mathrm{Rep}(G^\vee)\to\mathfrak z(\widehat\g)\mathrm{-mod},
$$
and our goal is to prove that it is $\sigma$-equivariant. According to \cite[1.9, 5.1-5.5]{Ra}, the resulting functor in our case equals the composition
$$
\mathrm{Rep}(G^\vee)\xrightarrow{\sim} 
\mathrm{Perv}_{G\jet t}(\mathrm{Gr}_G)\xrightarrow{(-)*\mc L^{-1}} D_{\mathrm{Gr}_G}\mathrm{-mod}\xrightarrow{\Gamma(\mathrm{Gr_G},-)} \widehat\g\mathrm{-mod}\xrightarrow{(-)^{G\jet t}}\mathfrak z(\widehat\g)\mathrm{-mod},
$$
where the first map is the geometric Satake equivalence and $\mc L$ is the critical level line bundle on $\mathrm{Gr}_G$. We define the $\sigma$-action on $\mathrm{Perv}_{G\jet t}(\mathrm{Gr}_G)$ and $D_{\mathrm{Gr}_G}\mathrm{-mod}$ by pullback. As pointed out in \cite[Section 4]{Ho}, the $\sigma$-action defined this way respects the tensor structure of these categories and makes the first map $\sigma$-equivariant.

Due to Lemma \ref{l:linebundlesareinvariant}, the line bundle $\mc L$ is $\sigma$-invariant, hence, using the fact that the convolution product is preserved by $\sigma$, the map given by the convolution with $\mc L$ is $\sigma$-equivariant. The other two maps are clearly $\sigma$-equivariant.

We follow \cite[5.6-5.11]{Ra} for the construction of the desired $B^\vee$-reduction. To construct such a one, it is enough to give a line bundle $l^\lambda$ for any dominant weight $\lambda$ of $G^\vee$ satisfying certain properties. Such bundle is defined as the $-\lambda(\rho)$-eigenspace of the operator $-t\partial_t$ on $\Gamma(\mathrm{Gr_G},V(\lambda)*\mc L^{-1})$. It is transported to the $-\sigma(\lambda)(\rho)$-eigenspace of the operator $-t\partial_t$ on $\Gamma(\mathrm{Gr_G},V(\sigma(\lambda))*\mc L^{-1})$ by $\sigma$, hence the whole system is $\sigma$-equivariant. It follows then that the constructed $B^\vee$-bundle is $\sigma$-equivariant, which finishes the proof.
\end{proof}
\begin{remark}
    {\rm
    The isomorphism $\mathrm{Spec}\,\mathfrak z(\widehat\g)\simeq\mathrm{Op}_{\g^\vee}(D)$ was originally constructed using the language of vertex algebras and Wakimoto modules, as described in \cite{F2}, \cite{F3}. We believe that it is possible to proceed through all the constructions made there and to convince oneself that they are all $\sigma$-equivariant.
    }
\end{remark}

\begin{corollary}
    The isomorphism $Z(\g)\simeq \mathrm{Fun}\,\op_{\g^\vee}(D^\times)$ is $\sigma$-equivariant.
\end{corollary}
\begin{proof}
    It follows from the proofs of \cite[Proposition 4.3.4, Lemma 4.3.5]{F3} that the application of the complete universal enveloping algebra functor to the isomorphism
    $$
\mathfrak z(\widehat\g)\simeq\mathrm{Fun}\mathrm{Op}_{\g^\vee}(D)
    $$
    is canonically isomorphic to the desired isomorphism in a $\sigma$-equivariant way.
\end{proof}
\begin{corollary}\label{cor:segalsugaweigenv}
    There exists a complete set of Segal-Sugawara vectors for $\mathfrak z(\widehat \g)$ all being eigenvectors of $\sigma$.
\end{corollary}
\begin{proof}
    We have a $\sigma$- and $\mathrm{Der}\,\C\jet t$-invariant identification
    $$
\mathrm{Spec}\,\mathfrak z(\widehat\g)\simeq \mathrm{Op}_{\g^\vee}(D)\simeq V_{\mathrm{can}}\jet t.
    $$
    Since $V_{\mathrm{can}}$ completely splits into $\sigma$-eigenspaces, the conclusion follows.
\end{proof}

Before proceeding to the next theorem, we define one more notion that we will use frequently. We will call the pair $\chi\in (\g^\vee)^*$ and $\chi'\in (\g_\sigma^\vee)^*$ {\bfseries compatible} if the elements $\chi$ and $\chi'$ are sent to the same element of $\h_\sigma\git W^\sigma\subset\h\git W$ via the canonical projections $\pi:(\g^\vee)^*\to \h\git W$ and $\pi:(\g_\sigma^\vee)^*\to \h_\sigma\git W^\sigma$. 
\begin{lemma}\label{l:densecomppairs}
    For any open dense subsets $U\subset (\g^\vee)^*$ and $V\subset (\g_\sigma^\vee)^*$ such that $U^\sigma\ne \varnothing$ there exists a compatible pair $(\chi,\chi')$ such that $\chi\in U^\sigma$ and $\chi'\in V$. 
\end{lemma}
\begin{proof}
    Since projections are open, then the images of $U$ and $V$ are open in $\h\git W$ and $\h_\sigma\git W^\sigma$, respectively. Since $U$ contains a $\sigma$-invariant element, then $\pi(U)\cap \h_\sigma\git W^\sigma$ is non-empty, which implies that this set intersects $\pi(V)$. This finishes the proof.
\end{proof}

In the next proposition, we will replace $\g^\vee$ by $\g$ to simplify the notation, but the statement and the proof will stay the same for $\g^\vee$.
\begin{proposition}\label{p:simplespecanddiagofgaudin}\
    \begin{itemize}
        \item[(i)] The subalgebra $\mc B_\chi(\g)$ is diagonalizable and has a simple spectrum on $V(\lambda)$ for generic $\chi\in\g^*$. In other words, for such $\chi$ there exists a basis of $V(\lambda)$, in which $\mc B^\lambda_\chi(\g)$ coincides with the algebra of diagonal matrices.
        \item[(ii)] There exists a $\sigma$-invariant element $\chi\in\g^*$ such that $\mc B_\chi(\g)$ is diagonalizable and has a simple spectrum on $V(\lambda)$.
    \end{itemize}
\end{proposition}
\begin{proof}
    Statement $(i)$ is essentially \cite[Corollary 4]{FFR}. We prove $(ii)$ in several steps.

    Choose a compact real form $\g^\R$ of $\g$ such that $\sigma$ acts on it, which is possible by \cite[Section 3, Proposition 7]{O} (I thank Mischa Elkner for this reference). Clearly, its fixed points coincide with $\g_\sigma^\R$, which is non-zero. According to \cite[Lemma 2]{FFR}, $\mc B_\chi(\g)$ acts diagonalizably on $V(\lambda)$ for any $\lambda$ and $\chi\in i\g^\R$. In particular, it is true for $\chi\in i\g_\sigma^\R$. Now, \cite[Corrolary 3]{FFR} implies that when $\chi$ is regular, $\mc B_\chi$ has a cyclic vector in $V(\lambda)$. Since $\g_\sigma^\R$ contains regular elements, then $i\g_\sigma^\R$ does. Thus, we can find a regular element $\chi\in i\g_\sigma^\R$ for which $\mc B_\chi$ is diagonalizable on $V(\lambda)$ and has a cyclic vector there. Therefore, $\mc B_\chi$ additionally has a simple spectrum on $V(\lambda)$, as desired.
\end{proof}

\begin{corollary}\label{cor:genericcomppair}
    There is an open dense subset $U\subset(\g_\sigma^\vee)^*$ such that for any $\chi'\in U$ the following assertions are satisfied:
\begin{itemize}
    \item $\mc B_{\chi'}(\g_\sigma^\vee)$ is diagonalizable and has a simple spectrum on any irreducible representation of $\g_\sigma^\vee$,
    \item There is a $\sigma$-invariant element $\chi\in (\g^\vee)^*$ such that $(\chi,\chi')$ is a compatible pair and $\mc B_{\chi}(\g^\vee)$ is diagonalizable and has a simple spectrum on any irreducible representation of $\g^\vee$.
\end{itemize}
\end{corollary}
\begin{proof}
    Easily follows from the last two results.
\end{proof}

Now, we are ready to state and prove our main result.
\begin{theorem}\label{t:isosofgaudin}\,
   \begin{enumerate}
       \item We have the following commutative diagram:
\begin{center}
{\small
    \begin{tikzcd}
Z_u(\widehat{\g^\vee})_\sigma \arrow[d, two heads]                            & \mathrm{Fun}\left(\mathrm{Op}_\g(D_u^\times)\right)^\sigma \arrow[l, "\sim"'] \arrow[d, two heads]  \arrow[r, "\sim"]                                   & \mathrm{Fun}\,\mathrm{Op}_{\g_\sigma}(D_u^\times) \arrow[r, "\sim"] \arrow[d, two heads]                                  & Z_u(\widehat{\g_\sigma^\vee}) \arrow[d, two heads]                          \\
\mathfrak z_u(\widehat{\g^\vee})_\sigma \arrow[d, two heads]                        & \mathrm{Fun}\left(\mathrm{Op}_\g(D_u)\right)^\sigma \arrow[d, two heads] \arrow[l, "\sim"']\arrow[r, "\sim"]                                           & {\mathrm{Fun}\,\mathrm{Op}_{\g_\sigma}(D_u)}  \arrow[r, "\sim"] \arrow[d, two heads]                                        & \mathfrak z_u(\widehat{\g_\sigma^\vee}) \arrow[d, two heads]\\
{\mathcal Z_{(z_i),\infty}(\g^\vee)_\sigma}\arrow[d, two heads]  & {\mathrm{Fun}\left(\mathrm{Op}_\g(\mathbb P^1)_{(z_i),\infty}\right)^\sigma} \arrow[l, "\sim"']\arrow[r, "\sim"]\arrow[d, two heads]& {\mathrm{Fun}\,\mathrm{Op}_{\g_\sigma}(\mathbb P^1)_{(z_i),\infty}} \arrow[r, "\sim"]\arrow[d, two heads] & {\mathcal Z_{(z_i),\infty}(\g_\sigma^\vee)}\arrow[d, two heads]    \\
{\mathcal Z_{(z_i),\infty}^{(m_i),m_\infty}(\g^\vee)_\sigma}  & {\mathrm{Fun}\left(\mathrm{Op}_\g(\mathbb P^1)_{(z_i),\infty}^{(m_i),m_\infty}\right)^\sigma} \arrow[l, "\sim"']\arrow[r, "\sim"]& {\mathrm{Fun}\,\mathrm{Op}_{\g_\sigma}(\mathbb P^1)_{(z_i),\infty}^{(m_i),m_\infty}} \arrow[r, "\sim"] & {\mathcal Z_{(z_i),\infty}^{(m_i),m_\infty}(\g_\sigma^\vee)} \makebox[0pt][l]{\,.}     
\end{tikzcd}
}
\end{center}

    \item Let $\chi,\chi'$ be a compatible pair with $\sigma(\chi)=\chi$. Then we have a commutative diagram
\begin{center}
{\small
    \begin{tikzcd}
{\mathcal B(\g^\vee)_\sigma} \arrow[d, two heads] & {\mathrm{Fun}\left(\mathrm{Op}_\g(\mathbb P^1)_{0,\infty}^{1,2}\right)^\sigma} \arrow[l, "\sim"'] \arrow[r, "\sim"]\arrow[d, two heads] & {\mathrm{Fun}\,\mathrm{Op}_{\g_\sigma}(\mathbb P^1)_{0,\infty}^{1,2}}  \arrow[d, two heads] \arrow[r, "\sim"] & {\mathcal B(\g_\sigma^\vee)} \arrow[d, two heads] \\
{\mathcal B_\chi(\g^\vee)_\sigma}\arrow[d, two heads]                 & {\mathrm{Fun}\left(\mathrm{Op}_\g(\mathbb P^1)_{0,\infty,\pi(-\chi)}^{1,2}\right)^\sigma} \arrow[l, "\sim"']\arrow[r, "\sim"]\arrow[d, two heads]                    & {\mathrm{Fun}\,\mathrm{Op}_{\g_\sigma}(\mathbb P^1)_{0,\infty,\pi(-\chi')}^{1,2}} \arrow[r, "\sim"]\arrow[d, two heads]                           & {\mathcal B_{\chi'}(\g^\vee)}\arrow[d, two heads]\\
{\mathcal B_\chi^\lambda(\g^\vee)_\sigma}                 & {\mathrm{Fun}\left(\mathrm{Op}_\g^{\mathrm{MF}}(\mathbb P^1)_{0,\infty,\pi(-\chi)}^{1,2,\lambda}\right)^\sigma} \arrow[l, two heads] \arrow[r, "\sim"]                   & {\mathrm{Fun}\,\mathrm{Op}_{\g_\sigma}^{\mathrm{MF}}(\mathbb P^1)_{0,\infty,\pi(-\chi')}^{1,2,\lambda}} \arrow[r,  two heads]                           & {\mathcal B_{\chi'}^\lambda(\g^\vee)}
 \makebox[0pt][l]{\,.}                
\end{tikzcd}
}
\end{center}
Moreover, for any $(\chi,\chi')$ as in Corollary \ref{cor:genericcomppair} the surjections in the bottom row are isomorphisms.

   \end{enumerate}

In all commutative diagrams above the maps between the functions on opers are induced by natural inclusions of the spaces of opers.
   
\end{theorem}
\begin{proof}
By Theorem \ref{t:actionsonffcenter}, $\sigma$ acts on both diagrams from Theorem \ref{t:isoswithopers}, which induce the left half of the diagrams from $(1)$ and $(2)$. Theorem \ref{t:sigma-invariantopers} gives the horizontal isomorphisms in the middle.
\end{proof}

\begin{remark}
{\rm
 One of the inclusions of opers in the diagrams from the previous theorem might not be a closed immersion of schemes, it is the inclusion
 $$
    \mathrm{Op}_\g(\mathbb P^1)_{(z_i),\infty}\hookrightarrow \op_\g(D_u).
 $$
 The reason is that $\mathrm{Op}_\g(\mathbb P^1)_{(z_i),\infty}$ is an ind-scheme and is not necessarily a scheme. Nevertheless, it is a union of schemes $\mathrm{Op}_\g(\mathbb P^1)_{(z_i),\infty}^{(m_i),m_\infty}$, which are closed subschemes of $\op_\g(D_u)$, whence the surjection between the spaces of functions.
 }
\end{remark}

\begin{remark}
{\rm
 Note that having the projections fixed, we can choose any $\chi$ and $\chi'$ in the same $G^\vee$ and $G_\sigma^\vee$ orbits, respectively, to construct a compatible pair. Suppose that we have pairs $(\chi,\chi')$ and $(g\cdot \chi,g'\cdot \chi')$. Then $\mc B_{\g\cdot\chi}=g\cdot \mc B_\chi$, where the action of $G^\vee$ comes from its action on $U(\g^\vee)$ and the embedding $\mc B_{\g\cdot\chi}(\g^\vee)\subset U(\g^\vee)$. This is compatible with the isomorphisms chosen, i.e. we have a commutative diagram
 \begin{center}
     \begin{tikzcd}
\mc B_\chi(\g^\vee) \arrow[d, "g\cdot"] \arrow[r,two heads] & \mc B_{\chi'}(\g_\sigma^\vee) \arrow[d, "g'\cdot"] \\
\mc B_{g\cdot\chi}(\g^\vee) \arrow[r,two heads]             & \mc B_{g'\cdot\chi'}(\g_\sigma^\vee))   \makebox[0pt][l]{\,.}        
\end{tikzcd}
 \end{center}
 }
\end{remark}

\subsection{Construction of remaining isomorphisms}
Still, Theorem \ref{t:isosofgaudin} doesn't answer Conjecture \ref{conj:isoofbigalg}. However, we can construct the desired map using it.
\begin{proposition}
    For any $\sigma$-invariant dominant coweight $\lambda$ of $\g$ there exists a unique epimorphism $\mc B^\lambda(\g^\vee)_\sigma\twoheadrightarrow\mc B^\lambda(\g_\sigma^\vee)$ compatible with the isomorphism $\mc B(\g^\vee)_\sigma\simeq \mc B(\g_\sigma^\vee)$.
\end{proposition}
\begin{proof}
    Take $\bar Q\in \mc B(\g^\vee)$ and let $Q$ and $Q'$ be its images in $\mc B^\lambda(\g^\vee)$ and $\mc B^\lambda(\g_\sigma^\vee)$, respectively. Our goal is to show that $Q=0$ implies $Q'=0$. The equality $Q=0$ implies that for every $\chi\in \h_\sigma$ we have $Q(\chi)=0$. Because of the existence of the isomorphism $\mc B_\chi^\lambda(\g^\vee)\twoheadrightarrow\mc B_\chi^\lambda(\g_\sigma^\vee)$ we also have $Q'(\chi')=0$ for any $\chi'$ forming a compatible pair with $\chi$, which on behalf of Corollary \ref{cor:genericcomppair} implies that $Q'=0$.
    \end{proof}

\begin{theorem}\label{t:solvetheconj}
    The map $\mc B^\lambda(\g^\vee)_\sigma\twoheadrightarrow\mc B^\lambda(\g_\sigma^\vee)$ is an isomorphism.
\end{theorem}
The proof will be given in Section \ref{s:transferatreps}.

\section{Concrete description of the isomorphism $\mathfrak z(\widehat{\g^\vee})_\sigma\simeq \mathfrak z(\widehat{\g_\sigma^\vee})$ and its derivatives}
In this section, we will give some results on the isomorphisms that have been obtained in the previous section. 
\subsection{The isomorphisms $Z(\g^\vee)_\sigma\simeq Z(\g_\sigma^\vee)$ and $S(\g^\vee)^{\g^\vee}_\sigma\simeq S(\g_\sigma^\vee)^{\g_\sigma^\vee}$}
We start with a recollection of the Harish-Chandra isomorphism. Let $\lambda$ be a dominant weight of $\g^\vee$ and $V(\lambda)$ be the corresponding highest weight representation. It is well-known that $Z(\g^\vee)$ acts on $V(\lambda)$ by scalars. Let 
$$
\chi_\lambda:Z(\g^\vee)\to \C
$$ 
be the algebra homomorphism such that every $P\in Z(\g^\vee)$ acts on $V(\lambda)$ by multiplication by $\chi_\lambda(P)$. Let $\rho^\vee$ be the Weyl vector of $\g^\vee$.

\begin{theorem}[Harish-Chandra isomorphism]
    Fix $P\in Z(\g^\vee)$. Then the map $\lambda+\rho^\vee\mapsto \chi_\lambda(P)$ extends to a $W$-invariant polynomial map $(\h^\vee)^*\to\C$. This defines a map 
    $$
    \gamma:Z(\g^\vee)\to S(\h^\vee).
    $$
    Then $\gamma$ is an isomorphism of $Z(\g^\vee)$ onto $S(\h^\vee)^W$.
\end{theorem}

We now construct an isomorphism 
$$
\phi:Z(\g^\vee)_\sigma\simeq Z(\g_\sigma^\vee)
$$
using the commutative diagram
\begin{center}
    \begin{tikzcd}                                                                 
Z(\g^\vee) \arrow[r, two heads]\arrow[d, "\simeq"] & Z(\g^\vee)_\sigma \arrow[d, "\simeq"] \arrow[r, "\phi", dashed]& Z(\g_\sigma^\vee) \arrow[d, "\simeq"]                     \\
S(\h^\vee)^W \arrow[d, "\simeq"] \arrow[r, two heads]                                           & S(\h^\vee)_\sigma^W \arrow[d, "\simeq"]                         & S(\h_\sigma^\vee)^{W^\sigma} \arrow[d, "\simeq"]                                      \\
S(\h^*)^W \arrow[r, two heads]                                                                  & S(\h^*)_\sigma^W \arrow[r, "\simeq"]                            & S(\h_\sigma^*)^{W^\sigma}  \makebox[0pt][l]{\,,}                                   
\end{tikzcd}
\end{center}
where the vertical isomorphisms are Harish-Chandra isomorphisms or identifications $\h^\vee\simeq \h^*$. The bottom isomorphism is induced by the restriction map $\h^*\to \h_\sigma^*$ and is an isomorphism due to the identification
$$
\spec S(\h^*)^W_\sigma\simeq (\h\git W)^\sigma=\h_\sigma\git W^\sigma\simeq \spec S(\h_\sigma^*)^{W^\sigma}
$$
from Proposition \ref{p:quotspacesisosfordynkin}. This isomorphism $\phi$ will be called {\bfseries canonical}.

Let $\lambda$ be a $\sigma$-invariant dominant weight of $\g^\vee$ and $V(\lambda)$ and $W(\lambda)$ be the corresponding highest weight representations of $\g^\vee$ and $\g_\sigma^\vee$.
\begin{proposition}
    We have the following commutative diagram:

\begin{center}
    \begin{tikzcd}
                                                                                                & \mathrm{End}(V(\lambda))                                            &                                                                              & \mathrm{End}(W(\lambda)) \\
U(\g^\vee) \arrow[ru]                                                                           & \C \arrow[u, hook] \arrow[rr, no head, equal, bend left]   & U(\g_\sigma^\vee) \arrow[ru]                                                      & \C \arrow[u,hook]         \\
Z(\g^\vee) \arrow[ru, "\chi_\lambda"'] \arrow[r, two heads] \arrow[u, hook]  & Z(\g^\vee)_\sigma \arrow[r, "\phi"] \arrow[u, dashed] & Z(\g_\sigma^\vee) \arrow[ru, "\chi_\lambda"'] \arrow[u, hook]\fullstopbelow &                      
\end{tikzcd}
\end{center}

In other words, there is a factorization of $\chi_\lambda$ through $Z(\g^\vee)_\sigma$ and for any $P\in Z(\g^\vee)$, $P$ acts on $V(\lambda)$ by the same scalar as the one by which $\phi(P)$ acts on $W(\lambda)$. Moreover, this property uniquely characterizes $\phi$.
\end{proposition}

\begin{proof}
We should explain the vertical dashed arrow and the commutativity of the rectangle including $\C$-s, $Z(\g^\vee)_\sigma$ and $Z(\g_\sigma^\vee)$. Note that for any $P\in Z(\g^\vee)$ and $\mu\in X(\g^\vee)$ we have
$$
\chi_{\sigma(\mu)}(\sigma P)\cdot \sigma(v)=\sigma(P)\cdot\sigma(v)=\sigma(P\cdot v)=\chi_\mu(P)\cdot \sigma(v)
$$
for any $v\in V(\mu)$. Therefore, 
$$
\chi_{\sigma(\mu)}(\sigma P)=\chi_\mu(P),
$$
which implies that for $\lambda\in X(\g_\sigma^\vee)=X(g^\vee)^\sigma$
$$
\chi_{\lambda}(P-\sigma P)=\chi_{\lambda}(P)-\chi_{\lambda}(\sigma P)=0,
$$
whence the factorization by the vertical dashed arrow.

The Weyl vector of $\g^\vee$ is fixed by $\sigma$ and coincides with the Weyl vector of $\g_\sigma^\vee$ (see Section~\ref{sec:dynkinonopers}), whence the following diagram: 
    \begin{center}
\begin{tikzcd}
                                                                             & \C \arrow[r, no head, equal]                            & \C &                                  & \gamma(P)(\lambda+\rho^\vee) \arrow[r, no head, equal] & \gamma(P)|_{\h_\sigma}(\lambda+\rho^\vee) \\
Z(\g^\vee)_\sigma \arrow[r] \arrow[d, "\gamma"'] \arrow[ru, "\chi_\lambda"] & Z(\g_\sigma^\vee) \arrow[d, "\gamma"'] \arrow[ru, "\chi_\lambda"] &    & P \arrow[r] \arrow[d] \arrow[ru] & \phi(P) \arrow[d] \arrow[ru]                      &                            \\
S(\h^\vee)_\sigma^W \arrow[r] \arrow[ruu]                   & S(\h_\sigma^\vee)^{W^\sigma} \arrow[ruu, "\mathrm{ev}_{\lambda+\rho^\vee}"']  &    & \gamma(P) \arrow[r] \arrow[ruu]  & \gamma(P)|_{\h_\sigma}\makebox[0pt][l]{\,.} \arrow[ruu]                     &                           
\end{tikzcd}
    \end{center}
The commutativity then follows.

Note that, due to the Harish-Chandra isomorphism, an element $\phi(P)$ is uniquely determined by actions of $\chi_\lambda$ on it for various $\lambda$. Therefore, the last statement follows.
\end{proof}

We can define an isomorphism $S(\g^\vee)^{\g^\vee}_\sigma\simeq S(\g_\sigma^\vee)^{\g_\sigma^\vee}$ in a similar way using the diagram

\begin{center}
    \begin{tikzcd}     
S(\g^\vee)_\sigma^{\g^\vee} \arrow[d, "\simeq"] \arrow[r, dashed]                         & S(\g_\sigma^\vee)^{\g_\sigma^\vee} \arrow[d, "\simeq"]                                      \\
S(\h^\vee)_\sigma^W \arrow[d, "\simeq"]                          & S(\h_\sigma^\vee)^{W^\sigma} \arrow[d, "\simeq"]                                      \\
   S(\h^*)_\sigma^W \arrow[r, "\simeq"]                            & S(\h_\sigma^*)^{W^\sigma}  \makebox[0pt][l]{\,.}                                   
\end{tikzcd}
\end{center}

We call this isomorphism {\bfseries canonical} too.
We then obtain a result similar to the previous one:
\begin{proposition}\label{p:canfors}
    The canonical isomorphism $\phi:S(\g^\vee)^{\g^\vee}_\sigma\simeq S(\g_\sigma^\vee)^{\g_\sigma^\vee}$ is uniquely determined by the property that for any compatible pair $(\chi,\chi')$ with $\sigma(\chi)=\chi$
    $$
    \phi(P)(\chi')=P(\chi),\qquad P\in S(\g^\vee)^{\g^\vee}.
    $$
\end{proposition}
\begin{proof}
    Follows directly from the construction.
\end{proof}
\subsection{Other isomorphisms and their compatibility}
\begin{proposition}\label{p:ffcentercompatibility}
    \begin{enumerate}
        \item[(i)] The composition $\mathfrak z(\widehat{\g^\vee})\to \mathfrak z(\widehat{\g^\vee})_\sigma\to\mathfrak z(\widehat{\g_\sigma^\vee})$ is compatible with gradings and commutes with the derivation $\tau=-\partial_t$.
        \item[(ii)] Let $S_1,\ldots,S_l\in\mathfrak z(\widehat{\g^\vee})$ be a complete set of Segal-Sugawara vectors which are all eigenvectors of $\sigma$. Let $S_1,\ldots,S_{l'}$ be the $\sigma$-invariant ones and $S_1',\ldots, S_{l'}'$ be their images in $\mathfrak z(\widehat{\g_\sigma^\vee})$. Then $S_1',\ldots, S_{l'}'$ is a complete set of Segal-Sugawara vectors in $\mathfrak z(\widehat{\g_\sigma^\vee})$.
    \end{enumerate}
\end{proposition}
\begin{proof}
    According to \cite[Proposition 5.3, Theorem 5.4, Theorem 5.7]{FFT} and our construction of maps between opers, all the maps in 

\begin{center}
{\small
    \begin{tikzcd}
\mathfrak z_u(\widehat{\g^\vee})_\sigma                      & \mathrm{Fun}\left(\mathrm{Op}_\g(D_u)\right)^\sigma \arrow[l, "\sim"']\arrow[r, "\sim"]                                           & {\mathrm{Fun}\,\mathrm{Op}_{\g_\sigma}(D_u)}  \arrow[r, "\sim"]                                        & \mathfrak z_u(\widehat{\g_\sigma^\vee}) \makebox[0pt][l]{\,.}     
\end{tikzcd}
}
\end{center}
are compatible with the actions of $\mathrm{Aut}\,\C\jet t$ and $\mathrm{Der}\,\C\jet t$. In particular, they are compatible with derivations $\tau=-\partial_t$ and $-t\partial_t$. The second one defines the gradings, so $(i)$ follows.

For $(ii)$, since 
$$
    \mathfrak z(\widehat{\g^\vee})_\sigma\simeq \C[\tau^kS_i:1\le i\le l, k\le 0]_\sigma\simeq \C[\tau^kS_i:1\le i\le l', k\le 0],
$$
we get from $(i)$ that 
$$
\mathfrak z(\widehat{\g_\sigma^\vee})\simeq \C[\tau^kS'_i:1\le i\le l', k\le 0],
$$
hence $S_1',\ldots,S_{l'}'$ is indeed a complete set of Segal-Sugawara vectors in $\mathfrak z(\widehat{\g_\sigma^\vee})$.
\end{proof}

\begin{lemma}\label{l:freepolyovercenters}
    The algebra $\mc B(\g^\vee)$ is a polynomial algebra over $Z(\g^\vee)$ and $S(\g^\vee)^{\g^\vee}$. Moreover, its algebraically independent generators can be chosen to be eigenvectors of $\sigma$.
\end{lemma}
\begin{proof}
    For this, we will use a general fact from linear algebra that whenever we have a polynomial $f$ in one variable $t^{-1}$ with coefficients in a vector space over $\C$, then the linear span of its coefficients equals the span of its values at $\deg f+1$ pairwise distinct values of $t$ and equals the span of values of $\partial_t^kf$ at a nonzero point and $\deg f+1$ values of $k$.
    
    Choose a complete set of Segal-Sugawara vectors $S_1,\ldots,S_l$. According to Proposition \ref{prop:genpropofgaudin}(iii), the images of elements $\tau^k S_i$ in $\mc B$, where $k=0\ldots \deg S_i$, form an algebraically independent subset of a set of generators of $\mc B$. We can also see from the description in Proposition \ref{prop:genpropofgaudin}(i) that, if we treat $z$ as a variable, then the map $\Psi$ will send $\tau$ to $\partial_z$. Therefore, applying the fact mentioned above, we see that the linear span of 
    $$
    \lge\tau^k S_i:k=0\ldots \deg S_i\rge
    $$
    equals the linear span of the coefficients of the image of $S_k$, considered as a polynomial in $z$. Choose this set of generators for each $i$. In particular, there will be generators of $S(\g^\vee)^{\g^\vee}$ in this set, being the coefficients at degree $0$, and generators of $Z(\g^\vee)$ being the coefficients at maximal degrees, as desired. Moreover, if the set $S_1,\ldots, S_l$ is chosen in such a way that $S_i$ are eigenvectors of $\sigma$ (which is always possible by Corollary \ref{cor:segalsugaweigenv}), then the generators of $\mc B(\g^\vee)$ constructed will be eigenvectors of $\sigma$.
\end{proof}
\begin{remark}
{\rm
    The argument above repeats the proof of \cite[Lemma 1]{R1}.
}
\end{remark}
\begin{theorem}\label{t:bigalgebrasiso}
    The isomorphism $\mc B(\g^\vee)_\sigma\simeq \mc B(\g_\sigma^\vee)$ has the following properties:
    \begin{itemize}
        \item[(i)] It preserves the gradings with respect to the polynomial part.
        \item[(ii)] It restricts to isomorphisms of subalgebras $Z(\g^\vee)_\sigma\simeq  Z(\g_\sigma^\vee)$ and $S(\g^\vee)^{\g^\vee}_\sigma\simeq S(\g_\sigma^\vee)^{\g_\sigma^\vee}$ which are the canonical ones.
    \end{itemize}
\end{theorem}
\begin{proof}
    We begin with $(i)$. Take a homogeneous element $P\in\mc B(\g^\vee)$ of degree $d$ and let $P'$ be its image in $\mc B(\g_\sigma^\vee)$. We want to find an element $\chi\in (\g_\sigma^\vee)^*$ such that $P(\chi)\ne 0$. If there is no such element, then $P=0$ on $(\g^\vee)_\sigma$, implying that $P'=0$ due to the commutative diagram
\begin{center}
    \begin{tikzcd}
\mc B(\g^\vee)_\sigma \arrow[r, "\simeq"] \arrow[d, two heads] & \mc B(\g_\sigma^\vee) \arrow[d, two heads] \\
\mc B_\chi(\g^\vee)_\sigma \arrow[r, "\simeq"]                 & \mc B_{\chi'}(\g_\sigma^\vee)            
\end{tikzcd}
\end{center}
for compatible pairs $(\chi,\chi')$. Thus, we may assume that we have found such an element $\chi$, hence a dense open subset of elements $\chi'\in (\g_\sigma^\vee)^*$ such that $P'(\chi)\ne 0$ and they form compatible pairs with $\sigma$-invariant elements $\chi\in (\g^\vee)^*$ such that $P(\chi)\ne 0$. It follows from the same commutative diagram that $P(\chi)$ is mapped to $P'(\chi')$. Then for any $t\in \C$ the element $P(t\chi)=t^dP(\chi)$ is mapped to $P'(t\chi')$, which implies that $P'(t\chi')=t^dP(\chi')$ for $\chi'$ from an open dense subset of $(\g_\sigma^\vee)^*$. Therefore, $P'$ is also homogeneous of degree $d$, as desired.

 We now switch to $(ii)$. Applying the above commutative diagram, we easily see that the isomorphism $\mc B(\g^\vee)_\sigma\simeq \mc B(\g_\sigma^\vee)$ restricts to the isomorphism of subrings $S(\g^\vee)^{\g^\vee}_\sigma\simeq S(\g_\sigma^\vee)^{\g_\sigma^\vee}$, as both these subrings can be characterized as the sets of elements which get constant after being evaluated at a generic element $\chi$. Moreover, it follows from Proposition \ref{p:canfors} that this isomorphism is canonical.

 For the remaining part of $(ii)$, we have $Z(\g^\vee)\subset \mc B(\g^\vee)$ and due to the fact that $\mc B(\g^\vee)$ is a free polynomial algebra over $Z(\g^\vee)$ with generators being $\sigma$-eigenvectors (Lemma \ref{l:freepolyovercenters}), there is an induced inclusion $Z(\g^\vee)_\sigma\subset \mc B(\g^\vee)_\sigma$. Since $Z(\g^\vee)$ is the zero-degree part of $\mc B(\g^\vee)$, we have maps

\begin{center}
\begin{tikzcd}
Z(\g^\vee)_\sigma \arrow[r, hook] \arrow[rd, hook] \arrow[rrr, dotted, bend left] & \mc B(\g^\vee)_\sigma \arrow[r, "\simeq"] \arrow[d, two heads] & \mc B(\g_\sigma^\vee) \arrow[d, two heads] & Z(\g_\sigma^\vee) \arrow[l, hook] \arrow[ld, hook] \\
                                                                                       & \mc B_\chi(\g^\vee)_\sigma \arrow[r, "\simeq"]                 & \mc B_{\chi'}(\g_\sigma^\vee)\commabelow                 &                                              
\end{tikzcd}
\end{center}
hence it is enough to restrict to $\mc B_{\chi'}(\g^\vee)$ for generic $\chi'$. Note that an element $P\in U(\g^\vee)$ lies in the center if and only if it acts diagonally on every representation. Since the algebra $\mc B_{\chi}^\lambda(\g^\vee)$ coincides with the diagonal matrices in $V(\lambda)$ for generic $\chi$, hence its center is the one-dimensional space spanned by the identity. As the identity goes to the identity, it follows that $P$ acts on $V(\lambda)$ by the same scalar as the image $\bar P$ of $P$ in $U(\g_\sigma^\vee)$ acts on $W(\lambda)$, which implies that we get the canonical map $Z(\g^\vee)_\sigma\xrightarrow{\sim} Z(\g_\sigma^\vee)$.
\end{proof}

Let us state one more compatibility that our isomorphisms satisfy. As described in \cite[Section 6.5]{Mo}, define a map 
\begin{align*}
\rho: U(\widehat{\g^\vee_-})&\to U(\g^\vee)\otimes \C[z^{-1}]\qquad X[r]\mapsto z^r X,\\
\rho_z: U(\widehat{\g^\vee_-})&\to U(\g^\vee)\qquad\qquad\qquad X[r]\mapsto z^r X.
\end{align*}

According to \cite[Proposition 6.5.2]{Mo}, the vectors $S_1,\ldots,S_{l'}$ are mapped to algebraically independent generators $P_1,\ldots,P_{l'}$ of $Z(\g^\vee)_\sigma$ and the vectors $S_1',\ldots,S_{l'}'$ are mapped to algebraically independent generators $P'_1,\ldots,P'_{l'}$ of $Z(\g_\sigma^\vee)$. Then we can define algebra isomorphisms between $Z(\g^\vee)_\sigma$ and $Z(\g_\sigma^\vee)$ by sending $P_i$ to $P'_i$ for each $z$, which due to the compatibility of $\rho$ with $\partial_t$ and $\partial_z$ will glue to a commutative diagram

\begin{center}
    \begin{tikzcd}
\mathfrak z(\widehat{\g^\vee})_\sigma \arrow[d, "\rho"] \arrow[r, "\sim"] & \mathfrak z(\widehat{\g_\sigma^\vee}) \arrow[d, "\rho"] \\
{Z(\g^\vee)_\sigma[z^{-1}]} \arrow[r, "\sim"]                             & {Z(\g_\sigma^\vee)[z^{-1}]}\makebox[0pt][l]{\,.}                             
\end{tikzcd}
\end{center}
\begin{proposition}
    The bottom map in this commutative diagram is induced by the canonical isomorphism $Z(\g^\vee)_\sigma\xrightarrow{\sim} Z(\g_\sigma^\vee)$.
\end{proposition}
\begin{proof}
    First of all, since Segal-Sugawara vectors $S_i$ are homogeneous and the morphisms are compatible with the grading, then the bottom map is indeed induced by an isomorphism $Z(\g^\vee)_\sigma\xrightarrow{\sim} Z(\g_\sigma^\vee)$.

    Note that the inclusion $Z(\g^\vee)\subset \mc B(\g^\vee)$ splits and there is a natural splitting map given by evaluation of the polynomial part at $0$. Thus, we have a commutative diagram

\begin{center}
    \begin{tikzcd}
\mathfrak z(\widehat{\g^\vee}) \arrow[r, "{\rho_{z,\infty}}"] \arrow[rd, "\rho_z"] & \mc B(\g^\vee) \arrow[d]          \\
                                                                         & Z(\g^\vee)\simeq \mc B_0(\g^\vee)\makebox[0pt][l]{\,,} 
\end{tikzcd}
\end{center}
hence the conclusion follows from Theorem \ref{t:bigalgebrasiso}(ii). 
\end{proof}

\section{Transfer at the level of representations}\label{s:transferatreps}
In this section, let $\lambda$ be a $\sigma$-invariant dominant coweight of $G$ and $\chi$ be a $\sigma$-invariant element of $\g^*$. We also denote the irreducible $G^\vee$-representation with the highest weight $\lambda$ by $V(\lambda)$ and the irreducible $G_\sigma^\vee$-representation with the highest weight $\lambda$ by $W(\lambda)$. We will consider the universal big algebra $\mc B$, its evaluation $\mc B^\lambda$ at the representation of highest weight $\lambda$, its evaluation $\mc B_\chi$ at $\chi$ and both evaluations $\mc B_\chi^\lambda$. 

This section is mainly devoted to the proof of Theorem \ref{t:solvetheconj}, solving Conjecture \ref{conj:isoofbigalg}. We will need several steps to achieve the goal.

\begin{lemma}\label{l:splitofmaps}
    The maps $\mc B\to\mc B_\sigma$, $\mc B^\lambda\to\mc B_\sigma^\lambda$, $\mc B_\chi\to\mc (B_\chi)_\sigma$, $\mc B^\lambda_\chi\to\mc (B^\lambda_\chi)_\sigma$ are split and the splittings commute with the natural surjections between these algebras.
\end{lemma}

\begin{proof}
    As all these maps are obtained as base changes of the map $\mc B\to\mc B_\sigma$, it is enough only to consider this case. It follows from Proposition \ref{cor:segalsugaweigenv} and Proposition \ref{prop:genpropofgaudin}(iii) that $\mc B$ is a polynomial algebra in variables being $\sigma$-eigenvectors. Write $\mc B=\C[x_1,\ldots,x_s,x_{s+1},\ldots,x_m]$, where $x_1,\ldots,x_s$ are the $\sigma$-invariant variables. Then $\mc B_\sigma=\C[x_1,\ldots,x_s]$, making the splitting obvious.
\end{proof}

Let $\C[X]$ denote the ring of regular functions on an algebraic variety $X$. Identifying $\g$ with $\g^*$ via an invariant nondegenerate bilinear form on $\g$ (e.g. the Killing form), we consider the restriction map to the Kostant section
$$
S(\g)\otimes \operatorname{End}V(\lambda)\simeq \C[\g]\otimes \operatorname{End}V(\lambda)\to \C[p_{-1}+\g^{p_1}]\otimes \operatorname{End}V(\lambda).
$$
\begin{lemma}\label{l:restrofinvtokostant}
    The induced map
    $$
    r:(S(\g)\otimes \operatorname{End}V(\lambda))^G\to \C[p_{-1}+\g^{p_1}]\otimes \operatorname{End}V(\lambda)
    $$
    is injective and commutes with evaluation at elements $\chi\in p_{-1}+\g^{p_1}$.
\end{lemma}
\begin{proof}
Take $A:=\sum_i P_i\otimes a_i\in (S(\g)\otimes \operatorname{End}V(\lambda))^G$. By definition, for any $\chi\in p_{-1}+\g^{p_1}$ we have $A(\chi)=r(A)(\chi)$, so the second part of the lemma is clear. Assume that $r(A)=0$. Then $A(\chi)=0$ for any $\chi\in p_{-1}+\g^{p_1}$. Take $g\in G$. Unravelling $g\cdot A=A$, we get
\begin{align*}
0=A(\chi)&=\sum_i P_i(\chi) a_i=\sum_i (g\cdot P_i)(\chi) (g\cdot a_i)\\
&=\sum_i P_i(g^{-1}\chi) (g\cdot a_i)=g\cdot A(g^{-1}\chi),
\end{align*}
which implies that $A(\chi)=0$ for any $\chi\in G\cdot (p_{-1}+\g^{p_1})$. As pointed out in the proof of Proposition \ref{p:quotspacesisosfordynkin}, each regular $G$-orbit intersects the Kostant section, hence $A(\chi)=0$ for $\chi$ from a Zariski-dense subset of $\g$, which implies that $A=0$.
\end{proof}

For simplicity, denote $S(\g^\vee)^{\g^\vee}$ and $S(\g_\sigma^\vee)^{\g_\sigma^\vee}$ by $R$ and $R_\sigma$, respectively. For $\chi\in(\g^\vee)^*$, we denote by $\C_\chi$ the $R$-module $\C$, on which each element $P\in R$ acts by multiplication by $P(\chi)$.
\begin{lemma}\label{l:basechangeofbig}
    Choose regular $\chi\in (\g^\vee)^*$. Then the canonical surjection
    \begin{align*}
        \mc B\otimes_R\C_\chi&\twoheadrightarrow\mc B_\chi
    \end{align*}
   is an isomorphism. If $\chi$ is generic regular and such that 
    \begin{equation}\label{eq:dimsagree}
        \dim \mc B^\lambda_\chi=\dim V(\lambda),
    \end{equation}
    then the surjection
    $$
    \mc B^\lambda\otimes_R\C_\chi\twoheadrightarrow\mc B^\lambda_\chi
    $$
    is an isomorphism. Moreover, if $(\chi,\chi')$ is a compatible pair with $\chi$ as above and $\sigma$-invariant, then the surjection
    $$
\mc B^\lambda_\sigma\otimes_{R}\C_\chi\simeq \mc B^\lambda_\sigma\otimes_{R_\sigma}\C_{\chi'}\twoheadrightarrow\mc (B^\lambda_\chi)_\sigma
    $$
    is also an isomorphism.
\end{lemma}
\begin{proof}
    For the first surjection, fix a set of algebraically independent generators $\mc B=\C[x_1,\ldots,x_m]$ such that $x_1,\ldots,x_l$ are algebraically independent generators of $R$, which is possible by Lemma \ref{l:freepolyovercenters}. Then 
    $$
    \mc B\otimes_R\C_\chi\simeq \C[x_{l+1},\ldots,x_m].
    $$
    On the other hand, the map $\mc B\to \mc B_\chi$ sends $x_{l+1},\ldots,x_m$ to algebraically independent generators of $\mc B_\chi$, as $\mc B_\chi$ has transcendence degree $m-l$ over $\C$ (Theorem \ref{t:freegensofbigalg}). Therefore, this map is an isomorphism, as desired.

    Using Lemma \ref{l:restrofinvtokostant}, embed $\mc B^\lambda$ into $\C[p_{-1}+\g^{p_1}]\otimes \operatorname{End}V(\lambda)$. Since the restriction 
    $$
    R=S(\g)^\g\simeq \C[\g]^\g\to \C[p_{-1}+\g^{p_1}]
    $$
    is an isomorphism (proof of Proposition \ref{p:quotspacesisosfordynkin}), we get an $R$-inclusion into a finite free $R$-algebra
    \begin{equation}\label{eq:bigintokostant}
        \mc B^\lambda\subset R\otimes \operatorname{End}V(\lambda).
    \end{equation}
    
    We know that $\mc B^\lambda_\chi$ has dimension $\dim V(\lambda)$ for generic $\chi$, hence any $\dim V(\lambda)+1$ elements of $\mc B^\lambda$ are linearly dependent over $R$. This implies that the localization $S_0^{-1}\mc B^\lambda$ at $S_0:=R\setminus \{0\}$ has dimension $\dim V(\lambda)$ as a vector space over $S_0^{-1}R$. Therefore, there is a nonzero element $a\in R$ and a free submodule $N\subset\mc B^\lambda$ of rank $\dim V(\lambda)$ such that $a\mc B^\lambda\subset N$. Let $\chi$ be a regular element satisfying (\ref{eq:dimsagree}) such that $a(\chi)\ne 0$ (this is precisely the genericity assumption needed in the lemma). As $a$ becomes a unit $a(\chi)\in \C_\chi$, we get
    $$
    \C^{\dim V(\lambda)}\simeq N\otimes_R \C_\chi=\mc B^\lambda\otimes_R \C_\chi\twoheadrightarrow \mc B^\lambda_\chi\simeq \C^{\dim V(\lambda)},
    $$
    which proves the claim.
    
    The third isomorphism is obtained by base change from the second one. 
\end{proof}

\begin{remark}
{\rm
    Note that the inclusion (\ref{eq:bigintokostant}) doesn't commute with evaluations at all $\chi\in\g\simeq \g^*$. We know that it commutes with evaluations at $\chi\in p_{-1}+\g^{p_1}$, but for $g\in G$ we have $\mc B^\lambda_{g\cdot\chi}=g\cdot \mc B^\lambda_\chi\ne \mc B^\lambda_\chi$, however the evaluations of $S(\g)^\g$ at $\chi$ and $g\cdot \chi$ coincide.
}
\end{remark}

\begin{remark}
{\rm
    According to \cite[Theorem 11]{Ko}, we have an $S(\g)^\g$-module (not an algebra) isomorphism
    $$
    (S(\g)\otimes \operatorname{End}V(\lambda))^\g\simeq S(\g)^\g\otimes \operatorname{End}_\h V(\lambda),
    $$
    so
    $$
    \mc B^\lambda\subset S(\g)^\g\otimes \operatorname{End}_\h V(\lambda)
    $$
    with $\mc B^\lambda_\chi$ being the image of $\mc B^\lambda$ in $\C_\chi\otimes \operatorname{End}_\h V(\lambda)$. After this observation, the proof of the second part of Lemma \ref{l:basechangeofbig} goes similarly without any usage of the Kostant section. 
}
\end{remark}

Now, we are finally prepared to prove Conjecture \ref{conj:isoofbigalg}.
\begin{proof}[Proof of Theorem \ref{t:solvetheconj}]
    Since there is a finite free $R$-module on the right-hand side of (\ref{eq:bigintokostant}), then $\mc B^\lambda$ is a finitely generated torsion-free module over $R$. The same is true for $\mc B^\lambda(\g_\sigma^\vee)$ as an $R_\sigma$-module, hence there exists a free $R_\sigma$-submodule $N\subset \mc B^\lambda(\g_\sigma^\vee)$ and a nonzero element $a\in R_\sigma$ such that $a\mc B^\lambda(\g_\sigma^\vee)\subset N$. Let $M$ be the preimage of $N$ under the map
    $$
\mc B^\lambda(\g^\vee)_\sigma\twoheadrightarrow\mc B^\lambda(\g_\sigma^\vee)
    $$
    and let $I$ be the kernel of this map. Then we have an exact sequence
    $$
    0\to I\cap M\to M\to N\to 0,
    $$
    which splits since $N$ is free. Then for any $\chi'\in (\g_\sigma^\vee)^*$ we have an exact sequence
    \begin{equation}
        0\to (I\cap M)\otimes \C_{\chi'}\to M\otimes_{S(\g_\sigma^\vee)^{\g_\sigma^\vee}}\C_{\chi'}\to N\otimes_{S(\g_\sigma^\vee)^{\g_\sigma^\vee}}\C_{\chi'}\to 0.\label{eq:exactaftertensor}
    \end{equation}

    Choose a generic regular $\chi'$ such that $a(\chi')\ne 0$. Since $a\mc B^\lambda(\g_\sigma^\vee)\subset N$, we have a commutative diagram 
    \begin{center}
        \begin{tikzcd}
N \arrow[d, "a\cdot"'] \arrow[r, hook] & B^\lambda(\g_\sigma^\vee) \arrow[d, "a\cdot"] \arrow[ld, dashed] \\
N \arrow[r, hook]                      & B^\lambda(\g_\sigma^\vee)                                       \fullstopbelow
\end{tikzcd}
    \end{center}
    Then we get a similar diagram after tensoring with $\C_\chi$, but with vertical arrows being multiplications by a nonzero complex number $a(\chi')$. Therefore, all the arrows become isomorphisms, so we get that
    $$
N\otimes_{S(\g_\sigma^\vee)^{\g_\sigma^\vee}}\C_{\chi'}\simeq B^\lambda(\g_\sigma^\vee)\otimes_{S(\g_\sigma^\vee)^{\g_\sigma^\vee}}\C_{\chi'}.
    $$
    Clearly, $a B^\lambda(\g^\vee)\subset M$, hence by the similar argument
$$
M\otimes_{S(\g_\sigma^\vee)^{\g_\sigma^\vee}}\C_{\chi'}\simeq B^\lambda(\g^\vee)_\sigma\otimes_{S(\g_\sigma^\vee)^{\g_\sigma^\vee}}\C_{\chi'}.
    $$
    It follows then from the isomorphism 
    $$
    \mc B_\chi^\lambda(\g^\vee)_\sigma\simeq\mc B_{\chi'}^\lambda(\g_\sigma^\vee)
    $$
    and Lemma \ref{l:basechangeofbig} that the epimorphism in (\ref{eq:exactaftertensor}) is an isomorphism. Therefore, $(I\cap M)\otimes \C_{\chi'}=0$. Let $I_{\chi'}$ be the ideal of $R_\sigma$ generated by $P-P(\chi')$ for all $P$. Then $I_{\chi'} (I\cap M)=I\cap M$. Note that since $\mc B^\lambda(\g^\vee)$ is finite over $R$, then $\mc B^\lambda(\g^\vee)_\sigma$ is finite over $R_\sigma$, hence $I$ is a finitely-generated $R_\sigma$-module. Therefore, the Cayley-Hamilton theorem implies that there exists an element $x\in I_{\chi'}$ such that $1+x$ annihilates $I\cap M$. According to Lemma \ref{l:splitofmaps}, we have a sequence of inclusions
    $$
    I\cap M\subset I\subset \mc B^\lambda(\g^\vee)_\sigma\subset \mc B^\lambda(\g^\vee),
    $$ 
    where the last algebra is torsion-free over $R_\sigma$. Clearly, $x\ne-1$, so we get that $I\cap M=0$. Then $aI\subset I\cap M=0$, so $I$ is annihilated by $a$. By the same argument, $I=0$, which finishes the proof.
\end{proof}

\begin{remark}
    {\rm
    It follows from \cite[Theorem 2.1 (ii)-(iii)]{Ha} that $\mc B^\lambda$ is a free $R$-module of rank $\dim V(\lambda)$. Thus, it is enough to take $N=\mc B^\lambda(\g_\sigma^\vee)$ in the above proof, and Lemma \ref{l:basechangeofbig} is true without the assumption of genericity of $\chi$.
    }
    
\end{remark}
Recall the isomorphisms
$$
\left(\mathrm{Spec}\,\mathcal B_\chi^\lambda(\g^\vee)\right)^\sigma\simeq \left(\mathrm{Op}_\g^{\mathrm{MF}}(\mathbb P^1)_{1,\infty,\pi(-\chi)}^{1,2,\lambda}\right)^\sigma\simeq \mathrm{Op}_{\g_\sigma}^{\mathrm{MF}}(\mathbb P^1)_{1,\infty,\pi(-\chi)}^{1,2,\lambda}\simeq \mathrm{Spec}\,\mathcal B_\chi^\lambda(\g^\vee).
$$
for a generic compatible pair $(\chi,\chi')$ with $\sigma(\chi)=\chi$. The set of joint eigenspaces of $\mathcal B_\chi^\lambda(\g^\vee)$ in  $V(\lambda)$ can be identified with $\mathrm{Spec}\,\mathcal B_\chi^\lambda(\g^\vee)$, hence we have a bijection
$$
\phi: \{\text{eigenvalues on }W(\lambda)\}\xrightarrow{\sim}\{\text{eigenvalues on } V(\lambda)\}^\sigma,
$$
which can be restated as giving a correspondence $\phi$ between eigenvalues such that the triangles

\begin{center}
    \begin{tikzcd}
\mc B_\chi^\lambda(\g^\vee)_\sigma \arrow[r, "\phi(\tau)"] \arrow[d, "\simeq"'] & \C \\
\mathcal B_\chi^\lambda(\g_\sigma^\vee) \arrow[ru, "\tau"']                            &   
\end{tikzcd}
\end{center}
commute.

We finish this section with a new proof of Jantzen's twining formula. I thank Leonid Rybnikov for his help with the argument.

\begin{lemma}\label{l:cycforallreg}
    Let $\chi\in\g$ be a regular element. Then the highest weight vector $v_\lambda$ of $V(\lambda)$ is a cyclic vector for $\mc B_\chi(\g)$.
\end{lemma}
\begin{proof}
    Identifying $\g^*$ with $\g$ via the Killing form, we may restate the proposition for $\chi\in \g$. Clearly, $v_\lambda$ is a cyclic vector for $\mc B_\chi$ if and only if it is a cyclic vector for $g\cdot \mc B_\chi=\mc B_{g\cdot \chi}$ for some $g\in G$, so we may assume that $\chi$ lies in the Kostant section $p_{-1}+\g^{p_1}$.

    Consider the automorphisms $\phi_s:=\exp(s)\cdot \exp(s\rho^\vee)$ of $\g$ for $s\in\R$. Write $\chi=p_{-1}+\tau$, where $\tau\in \g^{p_1}$. Then
    $$
    \phi_s(\chi)=\exp(s)\cdot \exp(s\rho^\vee)\cdot (p_{-1}+\tau)=p_{-1}+\exp(s)\cdot \exp(s\rho^\vee)\tau.
    $$
    It follows from \cite[(0.1.8)]{Ko} (or directly from the theory of $\mathfrak{sl}_2$-modules) that $\tau\in \n_+$, which implies that
    $$
    \lim_{s\to-\infty} \phi_s(\chi)=p_{-1}.
    $$
    
    Following the proofs of Corollary 2 and Corollary 3 in \cite{FFR}, we see that $v_\lambda$ is a cyclic vector for $\mc B_\chi$ for any $\chi$ from an open neighborhood of $p_{-1}$. Thus, there exists $s\in\R$ such that the statement is satisfied for $\chi':=\phi_s(\chi)\in p_{-1}+\g^{p_1}$. We see that
    $$
    \mc B_{\chi}=\mc B_{\phi_{-s}(\chi')}=\mc B_{\exp(-s\rho^\vee)\cdot\chi'}=\exp(-s\rho^\vee)\cdot\mc B_{\chi'},
    $$
    where the equality in the middle is the application of the general equality $\mc B_{t\cdot\chi}=\mc B_\chi$, which follows from Proposition \ref{prop:genpropofgaudin}(ii). Therefore, the statement is satisfied for $\mc B_\chi$, and we are done.
\end{proof}

Recall from Example \ref{e:imageofcasimir} that after identifying $\g^\vee\simeq (\g^\vee)^*$ via an invariant bilinear form, we get $\chi\in \mc B_\chi(\g^\vee)$. If we change the form by rescaling, then $\chi$ will be rescaled in the same way.
\begin{lemma}\label{l:imageofchi}
    Let $(\chi,\chi')$ be a compatible pair. Then the map
    $$
    \mc B_\chi(\g^\vee)\twoheadrightarrow\mc B_{\chi'}(\g_\sigma^\vee)
    $$
    sends $\chi$ to $\chi'$, where the identifications are made using the critical level forms on $\g^\vee$ and $\g_\sigma^\vee$. The same statement holds for Killing forms.
\end{lemma}
\begin{proof}
    Let $S_1\in \mathfrak z(\widehat{\g^\vee})$ and $S_1'\in \mathfrak z(\widehat{\g_\sigma^\vee})$ be as in Example \ref{e:imageofcasimir}, and let $\bar S_1$ and $\bar S_1'$be their images in $\mc B(\g^\vee)$ and $\mc B(\g_\sigma^\vee)$, respectively. 
    
    According to Proposition \ref{p:ffcentercompatibility}, $S_1$ is mapped to $cS_1'$, where $c$ is some nonzero complex number. Moreover, the same proposition tells that the graded pieces of $\bar S_1$ are mapped to $c$ times the corresponding graded pieces of $\bar S_1'$.
    
    Recall that the $0$-graded piece of $S_1$ equals $\sum X_iX^i\otimes 1\in Z(\g^\vee)\otimes 1\subset \mc B(\g^\vee)$. According to Theorem \ref{t:bigalgebrasiso}(ii), this element acts on each $V(\lambda)$ ($\sigma(\lambda)=\lambda$) by the same scalar as its image in $\mc B(\g^\vee)$ acts on $W(\lambda)$. Take $\lambda$ to be the highest weight of the adjoint representation. If the identifying forms are the Killing forms, then $\sum X_iX^i$ acts by $1$. Therefore, its image acts by $1$ as well, implying that $c=1$. Since the critical level forms equal $-1/2$ times the Killing forms, $c=1$ in this case too.
    
    This implies that the middle graded piece of $\bar S_1$ maps to the middle graded piece of $\bar S_1'$. But $\chi$ and $\chi'$ are the evaluations of the middle graded pieces at $\chi$ and $\chi'$, respectively. Since $(\chi,\chi')$ is a compatible pair, we are done.
\end{proof}
\begin{proof}[Proof of Theorem \ref{t:Jantzen}]
    % We start by showing the second part. We will show that $\dim \mc B_\chi^\lambda(\g^\vee)_\sigma=\operatorname{tr}(\sigma|V(\lambda))$ for $\sigma$-invariant regular $\chi\in U$, where $U$ is as in Corollary \ref{cor:genericcomppair}. These elements exist by the proof of Proposition \ref{p:simplespecanddiagofgaudin}, so we would have
    % $$
    % \dim W(\lambda)=\dim \mathcal B_\chi^\lambda(\g_\sigma^\vee)=\dim \mc B_\chi^\lambda(\g^\vee)_\sigma=\operatorname{tr}(\sigma|V(\lambda)).
    % $$
    
    Let $v_1,\ldots,v_n$ be a $\sigma$-equivariant basis of eigenvectors of $\mc B_\chi^\lambda(\g^\vee)$ and $x_1,\ldots,x_n$ be the corresponding idempotents in $\mc B_\chi^\lambda(\g^\vee)$. Let $a$ be the number of fixed $v_i$ and $b$ be the number of $v_i$ which are $\sigma$-eigenvectors with eigenvalue different from $1$. Note that 
    $$
        x_i(x_i-\sigma(x_i))=
    \begin{cases}
    cx_i,&\sigma(x_i)\ne x_i, c\in \C^*,\\
    0,&\sigma(x_i)= x_i,
    \end{cases}
    $$
    which implies that all $x_i$ such that $\sigma(x_i)\ne x_i$ are zero in $\mc B_\chi^\lambda(\g^\vee)_\sigma$. Thus, 
    $
    \dim \mc B_\chi^\lambda(\g^\vee)_\sigma=a.
    $
    We claim that $b=0$. This would imply that $\sigma$ viewed as a matrix in the basis $v_i$ has $1$ at the diagonal elements corresponding to fixed $v_i$-s and $0$ at all other diagonal elements, hence the epimorphism
    $$
    \mc B_\chi^\lambda(\g^\vee)\to\mc B_\chi^\lambda(\g_\sigma^\vee)
    $$
    would restrict to an isomorphism
    $$
    \operatorname{span}\{x_i:\sigma(x_i)=x_i\}\stackrel{\sim}{\to}\mc B_\chi^\lambda(\g_\sigma^\vee)
    $$
    and would send all $x_i$ with $\sigma(x_i)\ne x_i$ to zero. 
    % Moreover,
    % $$
    % \operatorname{tr}(\sigma|V(\lambda))=a=\dim \mc B_\chi^\lambda(\g^\vee)_\sigma.
    % $$
    In particular, if we assume that $v_1,\ldots, v_k$ is a basis of $V_\mu(\lambda)$ for a $\sigma$-invariant $\mu$, then
    \begin{equation}\label{eq:traceofweightspaces}
        \operatorname{tr}(\sigma|V_\mu(\lambda))=\#\{i=1\ldots k:\sigma(x_i)=x_i\}=\dim \sum_{i=1}^kx_i\mc B_\chi^\lambda(\g_\sigma^\vee).
    \end{equation}
    
    It follows from Lemma \ref{l:cycforallreg} that the highest vector $v_\lambda$ of $V(\lambda)$ is a cyclic vector for $\mc B_\chi^\lambda(\g^\vee)$. In addition, it is $\sigma$-invariant. Now, write $v_\lambda=\sum a_iv_i$. Since it is cyclic for $\mc B_\chi^\lambda(\g^\vee)$, we have $a_i\ne 0$ for every $i$. However, if there existed $j$ such that $\sigma(v_j)=\lambda v_j$ with $\lambda\ne 1$, then we would have
    $$
    \sum a_iv_i=v_\lambda=\sigma(v_\lambda)=\sum a_i\sigma(v_i),
    $$
    therefore $a_j=\lambda a_j$, a contradiction to $a_j\ne 0$. Therefore, $b=0$, and we are done with this part.

    Now, we prove the rest of the theorem. Choose a $\sigma$-invariant regular $\chi\in U\cap (\h_\sigma^\vee)^*$ such that $\chi$ `separates' the $\sigma$-orbits in weights of $V(\lambda)$, i.e. for any weights $\mu_1,\mu_2$ of $V(\lambda)$ not lying in the same $\sigma$-orbit, we have $\chi(\mu_1)\ne \chi(\mu_2)$. This is possible since there are only finitely many weights in $V(\lambda)$ and $U$ is open and dense in $(\g_\sigma^\vee)^*$. We assume that the basis $v_1,\ldots,v_n$ is chosen in such a way that $v_i$ are weight vectors in $V(\lambda)$.
    
    Let now $\mu$ be a $\sigma$-invariant weight and let $r:=\chi(\mu)$. Without loss of generality, assume that vectors $v_1,\ldots, v_k$ form a basis of $\ker(\chi-r)$. According to Lemma \ref{l:imageofchi}, $\chi$ corresponds to $\chi'\in \mc B_{\chi'}(\g_\sigma^\vee)$, which along with (\ref{eq:traceofweightspaces}) implies that
    \begin{align*}
        \operatorname{tr}(\sigma|V_\mu(\lambda))&=\#\{i=1\ldots k:\sigma(v_i)=v_i\}\\
        &=\#\{i=1\ldots n:\sigma(v_i)=v_i, v_i\in \ker(\chi-a)\}\\
        &=\dim\ker(\chi-r)\cdot B_{\chi'}(\g_\sigma^\vee)\\
        &=\dim\ker(\chi'-r)\\
        &=\dim W_\mu(\lambda).
    \end{align*}
    In these lines, by $\chi-r$ and $\chi'-r$ we mean the left multiplications by these elements, and the second and the last equalities follow from our assumption that $\chi$ separates $\sigma$-orbits of weights in $V(\lambda)$. So, we are done.
    % where $\tau$ is the weight of $W$ such that $\chi(\tau)=a/c$ (the corresponding weight space could be zero). In particular, for any weight $\tau$ of $W(\lambda)$ there is a $\sigma$-invariant weight $\mu$ of $V(\lambda)$ such that $\chi(\tau)=\chi(\mu)/c)$. Since $\chi$ separates $\sigma$-invariant weights, we have an equality of sets
    % $$
    % \{\chi(\mu):\sigma(\mu)=\mu\}=\{\chi(\mu)/c:\sigma(\mu)=\mu\},
    % $$
    % implying that $c=\pm 1$. If $c=1$, then we are done because in this case $\tau=\mu$. If $c=-1$, then $\tau=-\mu$ in this case. But in non-simply laced types, $-\mu=w_0\mu$ for the longest element $w_0$ of the Weyl group. Thus,
    % $$
    % \dim W_\tau(\lambda)=\dim W_\mu(\lambda),
    % $$
    % finishing the proof.
\end{proof}

\section*{Acknowledgments}
 The author is grateful to Tam{\'a}s Hausel for introducing him to the problem and Tam{\'a}s Hausel, Jakub L{\"o}wit, Nhok Tkhai Shon Ngo, Leonid Rybnikov and Oksana Yakimova for fruitful discussions.

\medskip
\noindent
{\bfseries Funding:} The research was conducted during the Scientific Internship in 2023 at the Institute of Science and Technology Austria. The author worked on related topics during the ISTernsip Summer Program in 2022 at the Institute of Science and Technology Austria. The author was supported by the Inicjatywa Doskonałości - Uniwersytet Jagielloński science scholarship for master students. The author was partially supported by the NCN grant SONATA NCN UMO-2021/43/D/ST1/02290.
\section*{Declarations}
\noindent
{\bfseries Conflict of Interest:} The author has no conflict of interest to declare that is relevant to this article.
\bibliographystyle{alphaurl}
\nocite{*}
\bibliography{references}

\end{document}